\documentclass{elsarticle}

\usepackage{graphicx}
\usepackage{amsmath}
\usepackage{amsthm}
\usepackage{color}
\usepackage{float}
\usepackage{multirow}

\theoremstyle{definition}
\newtheorem{thm}{Theorem}
\newtheorem{lem}[thm]{Lemma}

\newtheorem{assump}{Assumption}
\newtheorem{ex}{Example}
\newtheorem{rem}{Remark}

\numberwithin{equation}{section}

\begin{document}

\title{Error control of a numerical formula for the Fourier transform 
by Ooura's continuous Euler transform and fractional FFT}

\author[FUN]{Ken'ichiro Tanaka}
\ead{ketanaka@fun.ac.jp}
\address[FUN]{School of Systems Information Science, Future University Hakodate, \\
116-2 Kamedanakano-cho, Hakodate, Hokkaido, 041-8655, Japan}

\date{\today}

\begin{abstract}
In this paper, we consider a method for fast numerical computation of 
the Fourier transform of a slowly decaying function
with given accuracy in given ranges of the frequency.
In these decades, 
some useful formulas for the Fourier transform are 
proposed to recover difficulty of the computation due to 
the slow decay and the oscillation of the integrand. 
In particular, Ooura proposed formulas 
with continuous Euler transformation and showed their effectiveness. 
It is, however, also reported that errors of them become large 
outside some ranges of the frequency.
Then, 
for an illustrating representative of the formulas, 
we choose parameters in the formula based on its error analysis
to compute the Fourier transform
with given accuracy in given ranges of the frequency. 
Furthermore, combining the formula and fractional FFT, 
a generalization of the fast Fourier transform (FFT), 
we execute the computation in the same order of computation time as the FFT.
\end{abstract}


\maketitle

\section{Introduction}

In this paper, we consider a method for fast numerical computation of 
the Fourier transform 
\begin{align}
F(\omega) 
= \int_{-\infty}^{\infty} f(x)\, \mathrm{e}^{-\mathrm{i}\, \omega\, x}\, \mathrm{d}x
\label{eq:FT}
\end{align}
with given accuracy in given ranges of the frequency $\omega$.
In particular, we mainly focus on a function $f(x)$ in \eqref{eq:FT} 
with slow decay as $x \to \pm \infty$. 
The Fourier transform is a fundamental tool in various areas such as 
optics, signal processing, probability theory, 
and theoretical or numerical methods for differential equations \cite{bib:DuffyFTPDE2004}, etc.
Moreover, in mathematical finance, 
option pricing by the Fourier transform is an active research topic \cite{bib:Kwok_etal_2012}.
Particularly in this area, fast numerical computation of the Fourier transform is required. 
On the other hand, 
it is usually needed to compute 
the values of $F(\omega)$ in~\eqref{eq:FT}
for many $\omega$'s in some ranges.
Then, it is preferable to use some methods to accelerate the computation 
such as the fast Fourier transform (FFT).

Since the Fourier transform~\eqref{eq:FT} is a definite integral for fixed $\omega$, 
we may apply some standard quadrature formulas to~\eqref{eq:FT} 
such as Newton-Cotes type or Gauss type formulas. 
These formulas can yield accurate approximate values for $F(\omega)$
if the function $f(x)$ in~\eqref{eq:FT} is sufficiently smooth and decays rapidly as $x \to \pm \infty$.
For $f$ with slow decay, however, less accurate values are generated by these formulas. 
Moreover, we may also use double exponential (DE) formulas by 
Takahasi and Mori~\cite{bib:TakahasiMoriDE1974}
for the computation. 
The DE formulas are formed by some special variable transformations $\phi$ as 
\begin{align}
\int g(x)\, \mathrm{d}x \approx h \sum_{n=-N_{-}}^{N_{+}} g(\phi(nh))\, \phi'(nh)
\label{eq:AbstDE}
\end{align}
for a function $g$, where $h > 0$ is a width between sampling points.
The transformations $\phi$ are called double exponential (DE) transformations. 
The DE formulas are very accurate for a reasonably wide class of 
definite integrals~\cite{bib:KTanakaFuncClassDE2009}.
It is, however, known that the DE formulas 
do not yield so accurate results 
for oscillatory integrals such as~\eqref{eq:FT} with slowly decaying~$f$.

Several highly accurate formulas for the Fourier transform 
of slowly decaying functions are proposed by 
Ooura and Mori~\cite{bib:OouraMoriDE-FT1991}\cite{bib:OouraMoriDE-FT1999} and 
Ooura~\cite{bib:OouraDE-FT2005}.
In~\cite{bib:OouraMoriDE-FT1991}\cite{bib:OouraMoriDE-FT1999}, 
DE formulas specialized for oscillatory integrals are proposed. 
In these formulas, parameters such as $h$ in \eqref{eq:AbstDE} 
etc.~depend on the frequency $\omega$. 
Then, to compute $F(\omega)$ for another $\omega$, 
we need to change the parameters in these formulas.
In \cite{bib:OouraDE-FT2005}, 
another DE formula is proposed to compute $F(\omega)$ 
for a certain range of $\omega$ with constant parameters in the formula. 
This formula, however, has some restriction for the setting of the range, 
and direct application of the FFT to the formula is not straightforward. 
On the other hand, 
Ooura~\cite{bib:OouraEuler2001}\cite{bib:OouraEulerGen2003}
proposed other useful formulas
\begin{align}
F(\omega)
& \approx 
h \sum_{n=-N_{-}}^{N_{+}} w(|nh|)\, f(nh)\, \mathrm{e}^{-\mathrm{i}\, \omega\, nh}, 
\label{eq:TargetFormula_pre}
\end{align}
using continuous Euler transformations $w$ with some parameters. 
Accuracy of these formulas and 
applicability of the FFT to them are already shown. 
It is, however, also reported that errors of them become large 
for $\omega$ with large $| \omega |$ or 
for $\omega$ around points of discontinuity of $F$
when $h$ and the parameters in $w$ are independent of $\omega$.

In this paper, 
we focus on the formula \eqref{eq:TargetFormula_pre} in~\cite{bib:OouraEuler2001} 
with $w$ defined by~\eqref{eq:def_w} later. 
Then, 
based on error analysis of the formula, 
we find appropriate setting of the parameters in it
to compute the Fourier transform $F(\omega)$
with given accuracy in given ranges of the frequency $\omega$.
Furthermore, 
combining the formula and fractional FFT, a generalization of the FFT, 
we show that the computation can be done 
in the same order of computation time as the FFT.

The rest of this paper is organized as follows.
In Section~\ref{sec:Formula}, we describe the formula with the continuous Euler transform.
In Section~\ref{sec:ErrCont}, we investigate error of the formula, 
show appropriate setting of the parameters, 
and present an error bound of the formula under the setting.
In Section~\ref{sec:Num}, we explain the fractional FFT and show some numerical examples.
Proofs of lemmas for the error analysis are shown in Section~\ref{sec:Proofs}. 
Finally, we conclude this paper by Section~\ref{sec:Concl}.



\section{Numerical formula by Ooura's continuous Euler transform}
\label{sec:Formula}

Let $w(x; p, q)$ be defined by
\begin{align}
w(x; p, q) = \frac{1}{2} \mathop{\mathrm{erfc}} \left( \frac{x}{p} - q \right), 
\label{eq:def_w}
\end{align}
where 
\begin{align}
\mathop{\mathrm{erfc}}(x) = \frac{2}{\sqrt{\pi}} \int_{x}^{\infty} \exp(-t^{2})\, \mathrm{d}t.
\end{align}
Ooura \cite{bib:OouraEuler2001} introduced a continuous Euler transform of $F$ in~\eqref{eq:FT} 
defined by
\begin{align}
F_{w}(\omega) = \int_{-\infty}^{\infty} w(|x|; p, q) f(x)\, \mathrm{e}^{-\mathrm{i}\, \omega\, x} \mathrm{d}x.
\label{eq:F_w}
\end{align}
As shown later by Lemma~\ref{lem:E_w}, 
this function $F_{w}$ approximates $F$ on some assumptions. 
Then, applying the trapezoidal formula to the integral in~\eqref{eq:F_w} yields a numerical formula:
\begin{align}
F(\omega)
& \approx 
h \sum_{n=-N-1}^{N} w(|nh|; p, q) f(nh)\, \mathrm{e}^{-\mathrm{i}\, \omega\, nh}, 
\label{eq:TargetFormula}
\end{align}
where $h$ is a positive real number and $N$ is a positive integer. 
Using the approximation~\eqref{eq:TargetFormula}, 
we can numerically obtain the function $F$  
on a given interval $[-\omega_{u}, \omega_{u}]\ (\omega_{u} > 0)$. 
Namely, setting
\begin{align}
\tilde{h} = \omega_{u}/(N+1),
\label{eq:tilde_h}
\end{align}
we may compute the RHS of~\eqref{eq:TargetFormula} 
for $\omega = -(N+1)\, \tilde{h},\ldots , N \tilde{h}$, i.e., 
\begin{align}
h \sum_{n=-N-1}^{N} w(|nh|; p, q) f(nh)\, \mathrm{e}^{-\mathrm{i}\, m n h \tilde{h}} \qquad
(m = -N-1,\ldots, N).
\label{eq:TargetComputation}
\end{align}
To compute the values~\eqref{eq:TargetComputation} with given accuracy, 
we need to choose the parameters $p$, $q$, $h$ and $N$ appropriately. 
We give such choice of them in Section~\ref{sec:ErrCont}.

\section{Error control of the numerical formula}
\label{sec:ErrCont}

\subsection{Error estimate}

We begin with error estimate of the approximation~\eqref{eq:TargetFormula} 
to find appropriate choice of the parameters $p$, $q$, $h$ and $N$. 
Here, we introduce the following notations:
\begin{align}
F_{w}^{(\infty, h)}(\omega)
& =
h \sum_{n=-\infty}^{\infty} w(|nh|; p, q) f(nh)\, \mathrm{e}^{-\mathrm{i}\, \omega\, nh}, 
\label{eq:F_w_infty_h} \\
F_{w}^{(N, h)}(\omega)
& =
h \sum_{n=-N-1}^{N} w(|nh|; p, q) f(nh)\, \mathrm{e}^{-\mathrm{i}\, \omega\, nh}.
\label{eq:F_w_N_h}
\end{align}
Then, the error $ | F(\omega) - F_{w}^{(N, h)}(\omega) | $ can be bounded as follows:
\begin{align}
| F(\omega) - F_{w}^{(N, h)}(\omega) | 
\leq 
E_{w}(\omega) + E_{w}^{(\infty, h)}(\omega) + E_{w}^{(N, h)}(\omega),
\label{eq:PartTotalErr}
\end{align}
where
\begin{align}
E_{w}(\omega) &= | F(\omega) - F_{w}(\omega) |, \label{eq:Err1} \\
E_{w}^{(\infty, h)}(\omega) &= | F_{w}(\omega) - F_{w}^{(\infty, h)}(\omega) |, \label{eq:Err2} \\
E_{w}^{(N, h)}(\omega) &= | F_{w}^{(\infty, h)}(\omega) - F_{w}^{(N, h)}(\omega) |. \label{eq:Err3} 
\end{align}
To estimate these errors, we consider the domains
\begin{align}
\mathcal{T}_{\alpha} & = 
\{ z \in \mathbf{C} \mid 
| \mathrm{arg}\, z | < \mathop{\mathrm{arctan}} \alpha 
\text{ or } 
| \pi - \mathrm{arg}\, z | < \mathop{\mathrm{arctan}} \alpha \}, \\
\mathcal{D}_{d} &= \{ z \in \mathbf{C} \mid | \mathrm{Im}\, z | < d \} 
\end{align}
for  $0 < \alpha < 1$ and $d > 0$, 
and the following assumptions for $f$:
\begin{assump}
\label{assump:f_anal_etc}
The function $f$ is analytic on $\mathcal{T}_{\alpha}$, 
$| f(z) | \leq M$ for any $z \in \mathcal{T}_{\alpha}$, and
\begin{align}
\lim_{R \to \infty} \max_{-\alpha \leq \beta \leq \alpha}
\left|
f(\pm R + \mathrm{i}\, \beta\, R)
\right|
= 0.
\end{align}
\end{assump}
\begin{assump}
\label{assump:f_L^{2}}
The function $f$ is analytic on $\mathcal{D}_{d}$,
$| f(z) | \leq M$ for any $z \in \mathcal{D}_{d}$,
and belongs to $L^{2}(\mathbf{R})$, i.e., $f$ is square integrable on $\mathbf{R}$.
\end{assump}

First, an estimate of $E_{w}(\omega)$ is given by Lemma~\ref{lem:E_w} below. 
We omit the proof of this lemma because 
it can be shown in the same manner as Ooura~\cite[Theorem 2]{bib:OouraEuler2001}.

\begin{lem}
\label{lem:E_w}
On Assumption~\ref{assump:f_anal_etc}, for arbitrary $\alpha'$ 
with $0 < \alpha'\leq \alpha$, we have
\begin{align}
E_{w}(\omega)
\leq 
M \sqrt{1 + \alpha'^{2}} 
\left[
\frac{\sqrt{\pi} p}{\sqrt{1 - \alpha'^{2}}} 
\exp\left[ -q^{2} \left\{ 1 - \left( \frac{\alpha' p}{2q} |\omega| - 1 \right)^{2}/(1 - \alpha'^{2}) \right\} \right]
+
\frac{2}{\omega \alpha'}\, \mathrm{e}^{- |\omega|\, \alpha'\, p\, q}
\right]. 
\label{eq:ErrEst_final}
\end{align}
\end{lem}

Next, estimates of $E_{w}^{(\infty, h)}(\omega)$ and $E_{w}^{(N, h)}(\omega)$
are given by Lemmas~\ref{lem:DiscErr} and~\ref{lem:TrunErr} below, respectively.
Proofs of them are presented in Section~\ref{sec:Proofs}.

\begin{lem}
\label{lem:DiscErr}
On Assumption~\ref{assump:f_L^{2}}, 
for $\omega$ with $| \omega | \leq \pi / h$, 
we have
\begin{align}
E_{w}^{(\infty, h)}(\omega) 
\leq
\frac{2\, C_{M, p, q, d}}{1 - \exp(- 2\, \pi\, d/h)}
\exp \left( - \frac{\pi\, d}{h} \right), 
\label{eq:DiscErr}
\end{align}
where
\begin{align}
C_{M, p, q, d}
=
M
\left(
\frac{\sqrt{\pi}}{2} + q
\right)\, p\, \exp\{ (d / p)^{2} \}.
\label{eq:DiscErrConst}
\end{align}
\end{lem}

\begin{lem}
\label{lem:TrunErr}
On Assumption~\ref{assump:f_L^{2}} and under the condition $N h > pq$, we have
\begin{align}
E_{w}^{(N, h)}(\omega)
& \leq 
\frac{\sqrt{\pi}\, p\, M}{2}\, \exp \left\{ -\frac{(Nh- pq)^{2}}{p^{2}} \right\}. 
\label{eq:TrunErr}
\end{align}
\end{lem}

\begin{rem}
Lemmas~\ref{lem:DiscErr} and~\ref{lem:TrunErr} are important
because they give explicit error bounds,
although they are standard estimates 
of the discretization error and the truncation error of the trapezoidal formula 
for the function $w(|x|; p, q) f(x)\, \mathrm{e}^{-\mathrm{i}\, \omega\, x}$, respectively.
As a technical note, 
we point out that 
the function $w(|z|; p, q)$ is not an analytic function of $z$ in $\mathcal{D}_{d}$. 
Then, we apply a smoothing technique to $w(|z|; p, q)$ to prove Lemma~\ref{lem:DiscErr}
as shown in Section~\ref{sec:Proofs}.
\end{rem}

\begin{rem}
It is possible to use another weight function $\tilde{w}(z; \tilde{p}, \tilde{q})$ 
analytic in $\mathcal{D}_{d}$ which realizes a similar formula to~\eqref{eq:TargetFormula}. 
In this paper, however, we analyze the formula~\eqref{eq:TargetFormula}
to show theoretically that even the non-smooth function $w(|z|; p, q)$ 
can achieve the exponential convergence rate shown in Lemma~\ref{lem:DiscErr}.
\end{rem}

\subsection{Choice of parameters and error control}

Using Lemmas~\ref{lem:E_w}, \ref{lem:DiscErr} and~\ref{lem:TrunErr}, 
we show in Theorem~\ref{thm:FormulaParaErr} below 
that appropriate choice of the parameters 
enables the formula~\eqref{eq:TargetFormula}
to compute the Fourier transform $F(\omega)$ with 
given accuracy in given ranges of~$\omega$.

\begin{thm}
\label{thm:FormulaParaErr}
Let Assumptions~\ref{assump:f_anal_etc} and~\ref{assump:f_L^{2}} be satisfied. 
Let $\omega_{d}$ and $\omega_{u}$ be real numbers with $0 < \omega_{d} < \omega_{u}$
and $\omega_{d} / \omega_{u} \leq \min\{ \alpha, 1/2 \}$, 
let $N$ be an integer with
\begin{align}
N \geq \frac{2\, d\, (\omega_{d} + \omega_{u})\, \omega_{u}^{2}}{\pi \, \omega_{d}^{2}},
\label{eq:set_N}
\end{align}
and let $h$, $p$, $q$ be defined by
\begin{align}
h = \sqrt{\frac{2\, \pi\, d\, (\omega_{d} + \omega_{u} ) }{\omega_{d}^{2}\, N}}, \quad 
p = \sqrt{\frac{N\, h}{\omega_{d}}}, \quad
q = \sqrt{\frac{\omega_{d}\, N\, h}{4}}.
\label{eq:set_h_p_q}
\end{align}
Then, 
for any $\omega$ with $\omega_{d} \leq | \omega | \leq \omega_{u}$, 
we have 
\begin{align}
\left|
F(\omega) - F_{w}^{(N, h)}(\omega)
\right|
\leq 
C(N)\, \exp\left(- \sqrt{ \frac{\pi\, d\, \omega_{d}^{2}\, N }{2\, (\omega_{d} + \omega_{u} )} } \right),
\label{eq:MainEstim}
\end{align}
where $C(N) = C_{1}(N) + C_{2}(N) + C_{3}(N)$ and 
\begin{align}
C_{1}(N) 
& = 
M \sqrt{\omega_{u}^{2} + \omega_{d}^{2}}
\left\{ 
\frac{\sqrt{\pi}} {\sqrt{\omega_{u}^{2} - \omega_{d}^{2}}}
\left(\frac{2\, \pi\, d\, (\omega_{d} + \omega_{u})\, N}{\omega_{d}^{4}} \right)^{1/4}
+
\frac{2}{\omega_{d}^{2} }
\right\}, \label{eq:Const_1} \\
C_{2}(N)
& = 
\frac{2\, M}{1 - \mathrm{e}^{- 2\, d\, \omega_{u}}}
\left\{
\frac{\sqrt{\pi}}{2} 
\left(\frac{2\, \pi\, d\, (\omega_{d} + \omega_{u})\, N}{\omega_{d}^{4}} \right)^{1/4} 
+ \sqrt{\frac{\pi\, d\, (\omega_{d} + \omega_{u} ) N}{2\, \omega_{d}^{2}}}
\right\}\, \mathrm{e}^{d\, \omega_{d}/4},  
\label{eq:Const_2} \\
C_{3}(N)
&=
\frac{\sqrt{\pi}\, M}{2}
\left(\frac{2\, \pi\, d\, (\omega_{d} + \omega_{u})\, N}{\omega_{d}^{4}} \right)^{1/4}.
\label{eq:Const_3} 
\end{align}
\end{thm}

\begin{rem}
Since Assumption~\ref{assump:f_L^{2}} does not exclude the case $f \not \in L^{1}(\mathbf{R})$,
the Fourier transform $F$ may be discontinuous at $\omega = 0$. 
Then,  we consider the positive lower bound $\omega_{d}$ of $| \omega |$ 
in Theorem~\ref{thm:FormulaParaErr}.
\end{rem}

\begin{rem}
It follows from the estimate~\eqref{eq:MainEstim} 
that $ | F(\omega) - F_{w}^{(N, h)}(\omega) | = \mathrm{O}(\sqrt{N}\, \exp(-c \sqrt{N}))$
for a constant $c>0$ independent of $N$. 
Therefore, appropriate setting of~$N$ realizes any given accuracy of the formula.
\end{rem}

\begin{proof}[Proof of Theorem~\ref{thm:FormulaParaErr}]
Using the parameters in~\eqref{eq:set_h_p_q}, 
we estimate the error bounds in Lemmas~\ref{lem:E_w}, \ref{lem:DiscErr}, and~\ref{lem:TrunErr}.
First, setting $\alpha' = \omega_{d} / \omega_{u}$ in~\eqref{eq:ErrEst_final} of Lemma~\ref{lem:E_w} and 
noting $\omega_{d} \leq | \omega | \leq \omega_{u}$, 
we have
\begin{align}
1 - \left( \frac{\alpha' p}{2q} |\omega| - 1 \right)^{2}/(1 - \alpha'^{2}) 
& =
1 - \frac{ (| \omega | / \omega_{u} - 1 )^{2} }{ 1 - (\omega_{d} / \omega_{u})^{2} } 
\geq 
1 - \frac{  (\omega_{d} / \omega_{u} - 1 )^{2}}{ 1 - (\omega_{d} / \omega_{u})^{2} } 
= 
\frac{2\, \omega_{d} / \omega_{u}}{ 1 + \omega_{d} / \omega_{u} }. 
\end{align}
Therefore we have
\begin{align}
q^{2} 
\left\{
1 - \left( \frac{\alpha' p}{2q} |\omega| - 1 \right)^{2}/(1 - \alpha'^{2}) 
\right\}
\geq 
\frac{\omega_{d}^{2}\, N\, h}{2\, (\omega_{d} + \omega_{u} )}.
\label{eq:L1_1}
\end{align}
Moreover, as for the last term in~\eqref{eq:ErrEst_final}, we can deduce
\begin{align}
| \omega |\, \alpha'\, p\, q \geq \frac{\omega_{d}^{2}\, N\, h}{2\, \omega_{u}}. 
\label{eq:L1_2}
\end{align}
Then, combining~\eqref{eq:ErrEst_final}, \eqref{eq:L1_1} and~\eqref{eq:L1_2}, 
we have
\begin{align}
E_{w}(\omega) 
& \leq 
M \sqrt{1 + (\omega_{d} / \omega_{u})^{2}}
\left\{ 
\frac{\sqrt{\pi}} {\sqrt{1 - (\omega_{d} / \omega_{u})^{2}}}
\sqrt{\frac{N\, h}{\omega_{d}}}
+
\frac{2\, \omega_{u}}{\omega_{d}^{2} }
\right\} \, 
\exp\left(- \frac{\omega_{d}^{2}\, N\, h}{2\, (\omega_{d} + \omega_{u} )} \right) \notag \\
& =
C_{1}(N)\, \exp\left(- \sqrt{ \frac{\pi\, d\, \omega_{d}^{2}\, N }{2\, (\omega_{d} + \omega_{u} )} } \right). 
\label{eq:Est_L1}
\end{align}
Next, to use Lemma~\ref{lem:DiscErr} to estimate $E_{w}^{(\infty, h)}(\omega)$, 
we need to check $\omega_{u} \leq \pi / h$. 
This follows from~\eqref{eq:set_N} and the definition of $h$ in~\eqref{eq:set_h_p_q}. 
Then, substituting $h, p, q$ in~\eqref{eq:set_h_p_q} 
into~\eqref{eq:DiscErr}
and~\eqref{eq:DiscErrConst}, 
we have
\begin{align}
E_{w}^{(\infty, h)}(\omega) 
\leq
C_{2}(N)\, \exp \left( - \sqrt{ \frac{\pi\, d\, \omega_{d}^{2}\, N }{2\, (\omega_{d} + \omega_{u} )} } \right). 
\label{eq:Est_L2}
\end{align}
Finally, substituting $h, p, q$ in~\eqref{eq:set_h_p_q} into~\eqref{eq:TrunErr}
and noting $\omega_{d} / \omega_{u} \leq 1/2$, 
we have
\begin{align}
E_{w}^{(N, h)}(\omega)
& \leq
C_{3}(N)\, 
\exp \left( -\frac{\omega_{d}\, N\, h}{4} \right)
\leq
C_{3}(N)\, \exp \left( -\frac{\omega_{d}^{2}\, N\, h}{2\, (\omega_{d} + \omega_{u})} \right) \notag \\
& =
C_{3}(N)\, \exp \left( - \sqrt{ \frac{\pi\, d\, \omega_{d}^{2}\, N }{2\, (\omega_{d} + \omega_{u} )} } \right). 
\label{eq:Est_L3}
\end{align}
From the estimates~\eqref{eq:Est_L1}, \eqref{eq:Est_L2} and~\eqref{eq:Est_L3}, we obtain the conclusion.
\end{proof}

\section{Numerical Experiments}
\label{sec:Num}

Recall that we need to compute the values 
\begin{align}
F_{w}^{(N, h)}(m \tilde{h})
& =
h \sum_{n=-N-1}^{N} w(|nh|; p, q) f(nh)\, \mathrm{e}^{-\mathrm{i}\, m n h \tilde{h}}
\label{eq:F_w_N_h_m_tildeh}
\end{align}
for $m = -N-1,\ldots, N$, 
where $\tilde{h}$ is defined by~\eqref{eq:tilde_h} and 
$h$, $p$, $q$ are defined by~\eqref{eq:set_h_p_q}, respectively. 
Unless $\tilde{h}$ and $h$ satisfy $h \tilde{h} = \pi/(N+1)$, 
we cannot use the FFT directly to the computation of~\eqref{eq:F_w_N_h_m_tildeh}.
Then, we use fractional FFT described in Section~\ref{sec:FFFT} below, 
and show some numerical examples in Section~\ref{sec:NumEx}. 

\subsection{Fractional FFT}
\label{sec:FFFT}

Fractional FFT is developed by 
Bailey and Swarztrauber \cite{bib:BailSwarFRFT1991}
to enable computation of a sum such 
as~\eqref{eq:F_w_N_h_m_tildeh} with arbitrary $h$ and $\tilde{h}$.
For example, Chourdakis \cite{bib:ChourdakisPriceFracFFT2005}
applied the fractional FFT to option pricing problems.

The fractional FFT is derived by 
regarding the sum in~\eqref{eq:F_w_N_h_m_tildeh} as a circular convolution.
For simplicity, we introduce $m' = m + N + 1$ and $n' = n + N + 1$
to shift the range of the indexes as follows:
\begin{align}
F_{w}^{(N, h)}(m \tilde{h})
& =
h \sum_{n'=0}^{2(N+1)-1} w(|(n'-N-1)h|; p, q) f((n'-N-1)h)\, 
\mathrm{e}^{-\mathrm{i}\, (m'-N-1) (n'-N-1) h \tilde{h}} \notag \\
& =
\mathrm{e}^{\mathrm{i}\, (m'-N-1) (N+1) h \tilde{h}}
\sum_{n'=0}^{2(N+1)-1} x_{n'}\, \mathrm{e}^{-\mathrm{i}\, m'n' h \tilde{h}},
\label{eq:F_w_N_h_indexes_shift}
\end{align}
where 
$x_{n'} = h\, w(|(n'-N-1)h|; p, q) f((n'-N-1)h)\, \mathrm{e}^{\mathrm{i}\, (N+1)\, n'\, h \tilde{h}}$.
Then, for $m' = 0,\ldots, 2(N+1)$,
the sum in~\eqref{eq:F_w_N_h_indexes_shift} is rewritten in the form
\begin{align}
\sum_{n'=0}^{2(N+1)-1} x_{n'}\, \mathrm{e}^{-\mathrm{i}\, m'n' h \tilde{h}}
= 
\sum_{n'=0}^{2(N+1)-1} x_{n'}\, \mathrm{e}^{-\pi \mathrm{i}\, \{ m'^{2} + n'^{2} - (m' - n')^{2} \} \alpha}
= 
\mathrm{e}^{-\pi \mathrm{i}\, m'^2 \alpha } 
\sum_{n'=0}^{2(N+1)-1} y_{n'}\, z_{m' - n'},
\label{eq:conv1}
\end{align}
where $\alpha = h \tilde{h} / (2\pi)$, 
$y_{n'} = x_{n'}\, \mathrm{e}^{-\pi \mathrm{i}\, n'^2 \alpha}$ and
$z_{n'} = \mathrm{e}^{\pi \mathrm{i}\, n'^2 \alpha}$.
Note that the sum in \eqref{eq:conv1} is a convolution but not circular, i.e., 
we cannot regard $\{ z_{n'} \}_{n' = 0}^{2N+1}$ as $2(N+1)$-periodic such as $z_{n'+2(N+1)} = z_{n'}$.
Then, we need to convert this sum into a form of a circular convolution.
A way for this conversion is to extend
$\{ y_{n'} \}_{n'=0}^{2N+1}$ and $\{ z_{n'} \}_{n' = 0}^{2N+1}$ 
to length $4(N+1)$ by setting
\begin{align}
y_{n'} = 0,\quad 
z_{n'} = \mathrm{e}^{\pi \mathrm{i}\, \{ 4(N+1) - n' \}^2 \alpha}
\label{eq:yz_extend}
\end{align}
for $n'$ with $2(N+1) \leq n' < 4(N+1)$.
Then, 
combining~\eqref{eq:F_w_N_h_indexes_shift}, 
\eqref{eq:conv1} 
and~\eqref{eq:yz_extend}, 
we have an expression of $F_{w}^{(N, h)}(m \tilde{h})$ with a circular convolution:
\begin{align}
F_{w}^{(N, h)}(m \tilde{h})
=
e_{m'} 
\sum_{n'=0}^{4(N+1)-1} y_{n'}\, z_{m' - n'},
\label{eq:CircConv}
\end{align}
where 
$e_{m'} = \mathrm{e}^{2\pi\mathrm{i}\, (m'-N-1) (N+1) \alpha - \pi \mathrm{i}\, m'^2 \alpha } $.
The RHS of~\eqref{eq:CircConv} can be obtained by
\begin{align}
\{ e_{m'} \} \odot 
\left\{
\mathrm{IDFT}_{m'}
\left(\, 
\{ \mathrm{DFT}_{\ell'}(\{ y_{n'} \}) \} \odot 
\{ \mathrm{DFT}_{\ell'}(\{ z_{n'} \})\}\, 
\right)
\right\},
\label{eq:ExprDFT}
\end{align}
where $\odot$ denotes element-by-element vector multiplication, 
and 
$\mathrm{DFT}$ and $\mathrm{IDFT}$ denote 
the discrete and inverse discrete Fourier transform, respectively.
Therefore, 
applying the ordinal FFT to $\mathrm{DFT}$ and $\mathrm{IDFT}$ in~\eqref{eq:ExprDFT},
we can compute \eqref{eq:CircConv} 
with computation time $\mathrm{O}(N \log N)$.

\subsection{Numerical Examples}
\label{sec:NumEx}

We consider the following functions $f_{i}$ and their Fourier transforms $F_{i}$.

\begin{ex}
\label{ex:Ex1_K0}
\begin{align}
f_{1}(x) = \frac{1}{\sqrt{1 + x^{2}}}, \quad 
F_{1}(\omega) = 2\, K_{0}(|\omega|).
\label{eq:Ex1}
\end{align} 
\end{ex}

\begin{ex}
\label{ex:Ex2_Gamma}
\begin{align}
f_{2}(x) = \frac{1}{(1 - \mathrm{i}\, x)^{2}}, \quad 
F_{2}(\omega) = 
\begin{cases}
0 & (\omega < 0),\\
2\pi\, \omega\, \mathrm{e}^{-\omega} & (\omega \geq 0).
\end{cases}
\label{eq:Ex2}
\end{align} 
\end{ex}

\noindent
The function $f_{1}$ is taken from~\cite{bib:OouraEuler2001}, 
where $K_{0}$ is the modified Bessel function of the second kind. 
The function $f_{2}$ is the characteristic function of the gamma distribution $\mathrm{Ga}(2,1)$, 
and $F_{2}(\omega)/(2\pi)$ is the density function of $\mathrm{Ga}(2,1)$.
The functions $f_{1}$ and $f_{2}$ satisfy
$f_{1}(x) = \mathrm{O}(|x|^{-1})\ (x \to \pm \infty)$ and 
$f_{2}(x) = \mathrm{O}(|x|^{-2})\ (x \to \pm \infty)$, respectively.
Moreover, 
$F_{1}$ is discontinuous at $\omega = 0$, 
and $F_{2}$ is not differentiable at $\omega = 0$.

To use Theorem~\ref{thm:FormulaParaErr}, 
we note that $f_{1}$ satisfies Assumptions~\ref{assump:f_anal_etc} and~\ref{assump:f_L^{2}} for 
\begin{align}
\alpha = d = 1 - \varepsilon_1,\ M = 1/\sqrt{\varepsilon_{1}}, 
\label{eq:assump_para_f1}
\end{align}
and $f_{2}$ satisfies them for 
\begin{align}
\alpha = d = 1 - \varepsilon_2,\ M =  \max\{ 2,\, 1/\varepsilon_2^{2} \},
\label{eq:assump_para_f2}
\end{align}
where $\varepsilon_{i}\ (i = 1,2)$ are arbitrary real numbers with $0< \varepsilon_{i} <1\ (i=1,2)$.
Here, we use $\varepsilon_{1} = 0.01$ and $\varepsilon_{2} = 0.1$. 
Using these, we apply the formula~\eqref{eq:F_w_N_h_m_tildeh} to $f_{i}\ (i=1,2)$
in the following procedure.
\begin{enumerate}
\item Give $\omega_{d}$ and $\omega_{u}$ in Theorem~\ref{thm:FormulaParaErr}, and 
set an error bound $\epsilon$ by which errors of the computed values for $\omega = m \tilde{h}$ 
with $\omega_{d} \leq |\omega| \leq \omega_{u}$ should be bounded. 
\item Choose $N$ from $\{ 2^{j} - 1 \mid j = 1,2,\ldots \}$ satisfying~\eqref{eq:set_N} and 
$\text{(The RHS of \eqref{eq:MainEstim})} \leq \epsilon$.
Let $N_{\epsilon}$ denote $N$ chosen here.
\item Set $h$, $p$ and $q$ by~\eqref{eq:set_h_p_q}, and apply the formula~\eqref{eq:F_w_N_h_m_tildeh}
to $f_{i}\ (i=1,2)$ with the fractional FFT.
\end{enumerate}
For this experiment, we choose the following sets of $\omega_{d}$, $\omega_{u}$ and $\epsilon$.
\begin{align}
& \text{(A)}\ \omega_{d} = 2, \omega_{u} = 10, \quad
\text{(B)}\ \omega_{d} = 1, \omega_{u} = 10, \quad 
\text{(C)}\ \omega_{d} = 1.25, \omega_{u} = 15, 
\label{eq:RangeSet} \\
& \text{(a)}\ \epsilon = 10^{-3}, \quad 
\text{(b)}\ \epsilon = 10^{-6}.
\label{eq:ErrSet}
\end{align}
Then, we compute~\eqref{eq:F_w_N_h_m_tildeh} for 
the six combinations of~\eqref{eq:RangeSet} and~\eqref{eq:ErrSet} for $f_{i}\ (i=1,2)$.

All computations are done by MATLAB R2013a programs with double precision floating point arithmetic
on a PC with 3.0GHz CPU and 2GB RAM.
The integers $N_{\epsilon}$ and computation times in the all cases are shown 
by Table~\ref{tab:N_e_and_Time_Ex1} for Example~\ref{ex:Ex1_K0} 
and
Table~\ref{tab:N_e_and_Time_Ex2} for Example~\ref{ex:Ex2_Gamma}. 
Each computation time is a mean of 
three results measured by three executions of each case.
These results show that 
the computation time is about doubled
when $N_{\epsilon}$ increases about by two times, 
which is about in accordance with the order of the theoretical time $\mathrm{O}(N \log N)$. 
Furthermore, errors of these computations are shown by 
Figures~\ref{fig:Ex1_02_10_m3}--\ref{fig:Ex1_01_15_m6}
for Example~\ref{ex:Ex1_K0} and 
Figures~\ref{fig:Ex2_02_10_m3}--\ref{fig:Ex2_01_15_m6}
for Example~\ref{ex:Ex2_Gamma}, respectively.
Then, we can observe that these results are consistent with 
the theoretical estimate of Theorem~\ref{thm:FormulaParaErr}, 
although it does not seem to be tight 
because the real errors are much smaller than the bounds given in~\eqref{eq:ErrSet}.

Here, using Example~\ref{ex:Ex2_Gamma}, 
we show another application of the formula~\eqref{eq:F_w_N_h_m_tildeh}.
Let $G_{2}$ denote the indefinite integral of $F_{2}(\omega)/(2\pi)$:
\begin{align}
G_{2}(\omega)
=
\frac{1}{2\pi} \int_{-\infty}^{\omega} F_{2}(\nu)\, \mathrm{d}\nu
=
\begin{cases}
0 & (\omega < 0),\\
1- (1+\omega)\, \mathrm{e}^{-\omega} & (\omega \geq 0).
\end{cases}
\label{eq:GaDist}
\end{align}
This is the cumulative distribution function of the Gamma distribution $\mathrm{Ga}(2,1)$.
We show that $G_{2}$ can be computed directly from the characteristic function $f_{2}$ 
in~\eqref{eq:Ex2} by the formula~\eqref{eq:F_w_N_h_m_tildeh}.
Since $G_{2}$ does not have the inverse Fourier transform, 
using the Heaviside function
\begin{align}
H(\omega)
=
\begin{cases}
0 & (\omega < 0), \\
1 & (\omega \geq 0),
\end{cases}
\end{align}
we define a function $\tilde{G}_{2}$ by
\begin{align}
\tilde{G}_{2}(\omega) = G_{2}(\omega) - H(\omega)
\end{align}
such that $\tilde{G}_{2}$ has the inverse Fourier transform $\tilde{f}_{2}$:
\begin{align}
\tilde{f}_{2}(x)
=
\frac{1}{2\pi}
\int_{-\infty}^{\infty} \tilde{G}_{2}(\omega)\, \mathrm{e}^{\mathrm{i}\, x\, \omega}\, \mathrm{d}\omega 
= 
\begin{cases}
\displaystyle
\mathrm{i}\, \frac{f_{2}(x) - 1}{ 2 \pi x}  & (x \neq 0), \\
\mathrm{i}\, f'_{2}(0) / (2\pi) & (x = 0).
\end{cases}
\label{eq:trans_H_2}
\end{align}
The function $\tilde{f}_{2}$ satisfies 
Assumptions~\ref{assump:f_anal_etc} and~\ref{assump:f_L^{2}} 
for 
\begin{align}
\alpha = d = 1 - \tilde{\varepsilon}_{2},\, 
M = \max\{ 2\sqrt{2}/\pi, 3/(2\pi \tilde{\varepsilon}_{2}^{2}) \}, 
\label{eq:ParaEx2plus}
\end{align}
where $\tilde{\varepsilon}_{2}$ is an arbitrary real number with $0< \tilde{\varepsilon}_{2} <1$.
Then, using the formula~\eqref{eq:F_w_N_h_m_tildeh}, 
we can obtain approximate values of $G_{2}(m \tilde{h})\ (m=-N-1,\ldots, N)$ as follows:
\begin{align}
G_{2}(m \tilde{h})
\approx
h \sum_{n=-N-1}^{N} w(|nh|; p, q) \tilde{f}_{2}(nh)\, \mathrm{e}^{-\mathrm{i}\, m n h \tilde{h}} 
+ H(m \tilde{h}).
\label{eq:G_2_approx}
\end{align}
Here, we show the result in the case of (A)-(a) in~\eqref{eq:RangeSet} and~\eqref{eq:ErrSet}.
The other cases can be done similarly. 
For $\tilde{\varepsilon}_{2} = 0.1$ in~\eqref{eq:ParaEx2plus} 
we choose $N_{\epsilon} = 1023$ and obtain the result shown in Figure~\ref{fig:Ex3_02_10_m3}.

\begin{table}[t]
\begin{center}
\caption{The values of $N_{\epsilon}$ determined in the step 2 of the procedure 
and computation times for Example~\ref{ex:Ex1_K0}.}
\label{tab:N_e_and_Time_Ex1}
\begin{tabular}{c c  r  c }
 & & $N_{\epsilon}$ & Time (sec.) \\
\hline
\multirow{2}{*}{(A)} & (a) &   $511$ & $0.039$ \\
                         & (b) & $1023$ & $0.072$ \\
\multirow{2}{*}{(B)} & (a) & $2047$ & $0.138$ \\
                         & (b) & $4095$ & $0.271$ \\
\multirow{2}{*}{(C)} & (a) & $2047$ & $0.137$ \\
                         & (b) & $4095$ & $0.268$ \\
\end{tabular}
\end{center}
\end{table}

\begin{figure}[H]
\begin{center}
\begin{minipage}{0.45\linewidth}
\includegraphics[width=.9\linewidth]{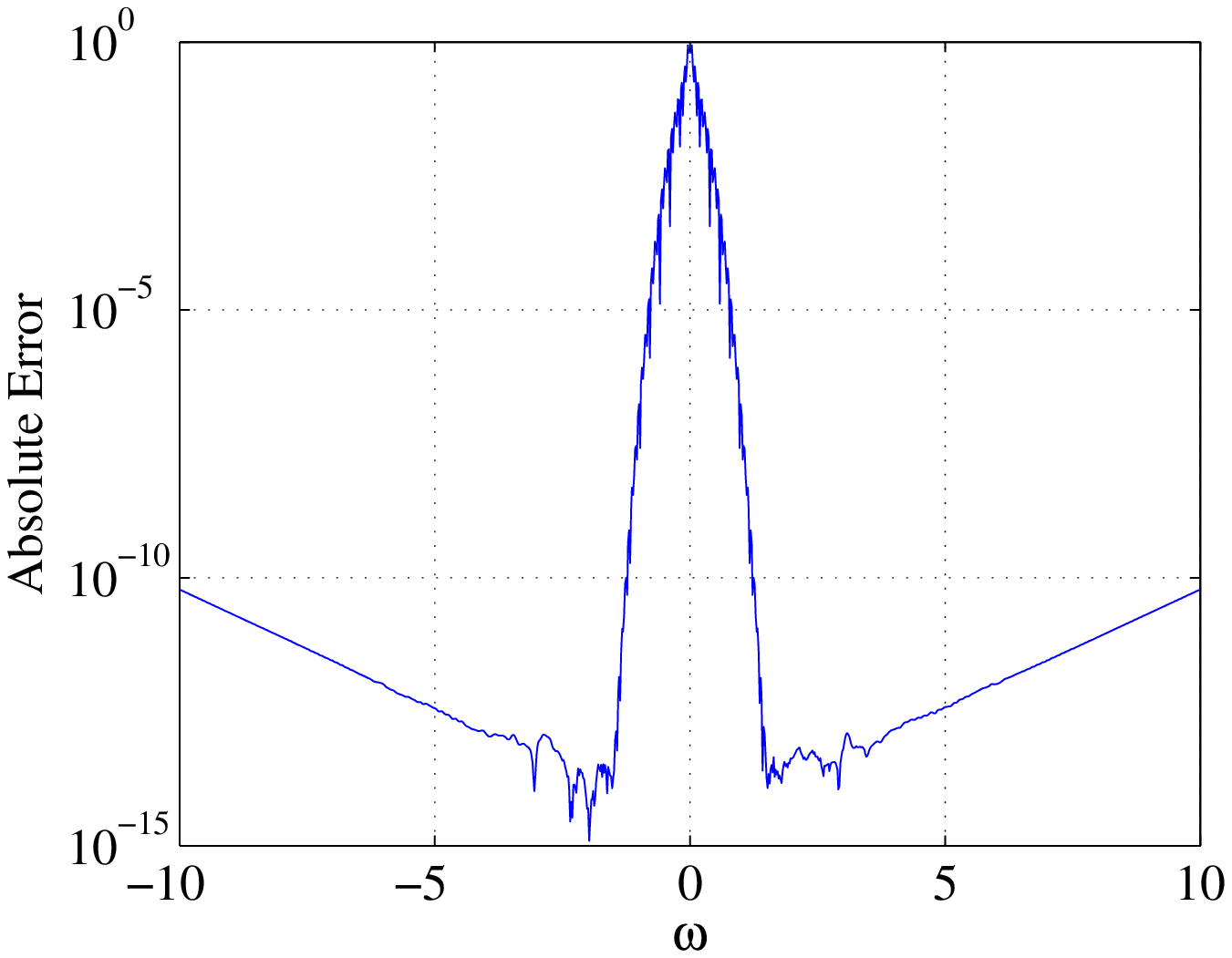}
\caption{Error for $F_{1}$ in~\eqref{eq:Ex1} 
in the case of (A)-(a) in~\eqref{eq:RangeSet} and~\eqref{eq:ErrSet}.}
\label{fig:Ex1_02_10_m3}
\end{minipage}
\quad  
\begin{minipage}{0.45\linewidth}
\includegraphics[width=.9\linewidth]{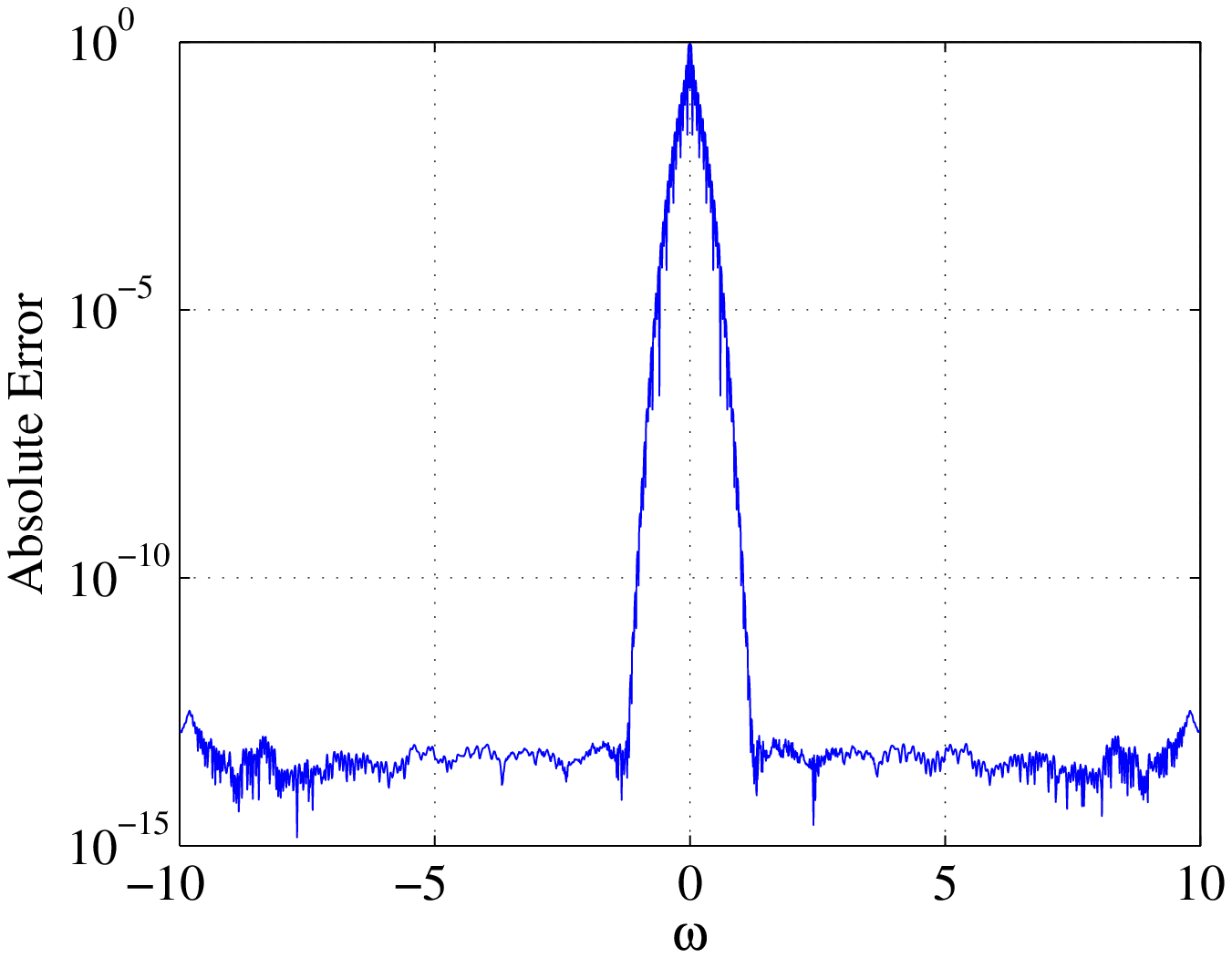}
\caption{Error for $F_{1}$ in~\eqref{eq:Ex1} 
 in the case of (A)-(b) in~\eqref{eq:RangeSet} and~\eqref{eq:ErrSet}.}
\label{fig:Ex1_02_10_m6}
\end{minipage}

\begin{minipage}{0.45\linewidth}
\includegraphics[width=.9\linewidth]{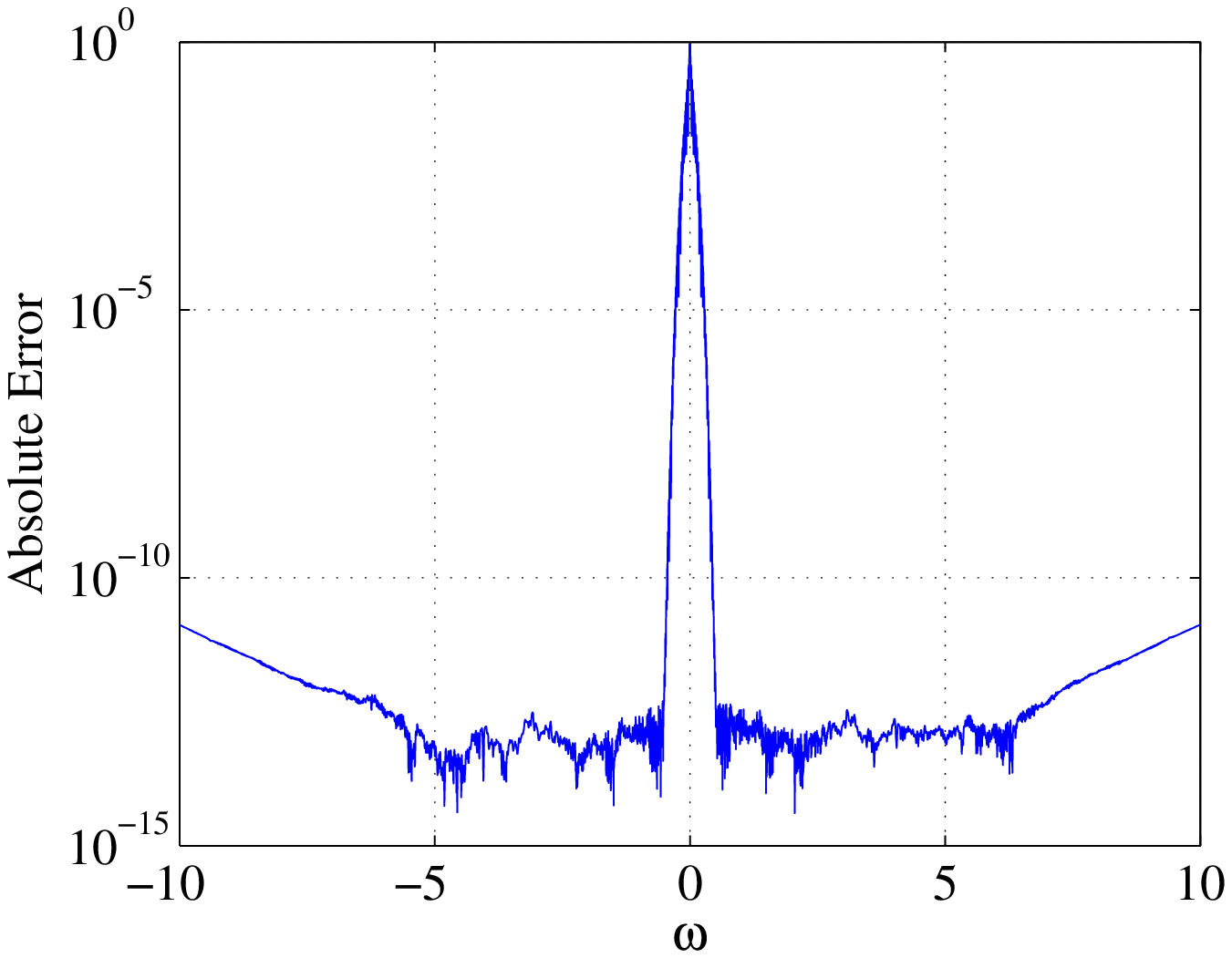}
\caption{Error for $F_{1}$ in~\eqref{eq:Ex1} 
 in the case of (B)-(a) in~\eqref{eq:RangeSet} and~\eqref{eq:ErrSet}.}
\label{fig:Ex1_01_10_m3}
\end{minipage}
\quad  
\begin{minipage}{0.45\linewidth}
\includegraphics[width=.9\linewidth]{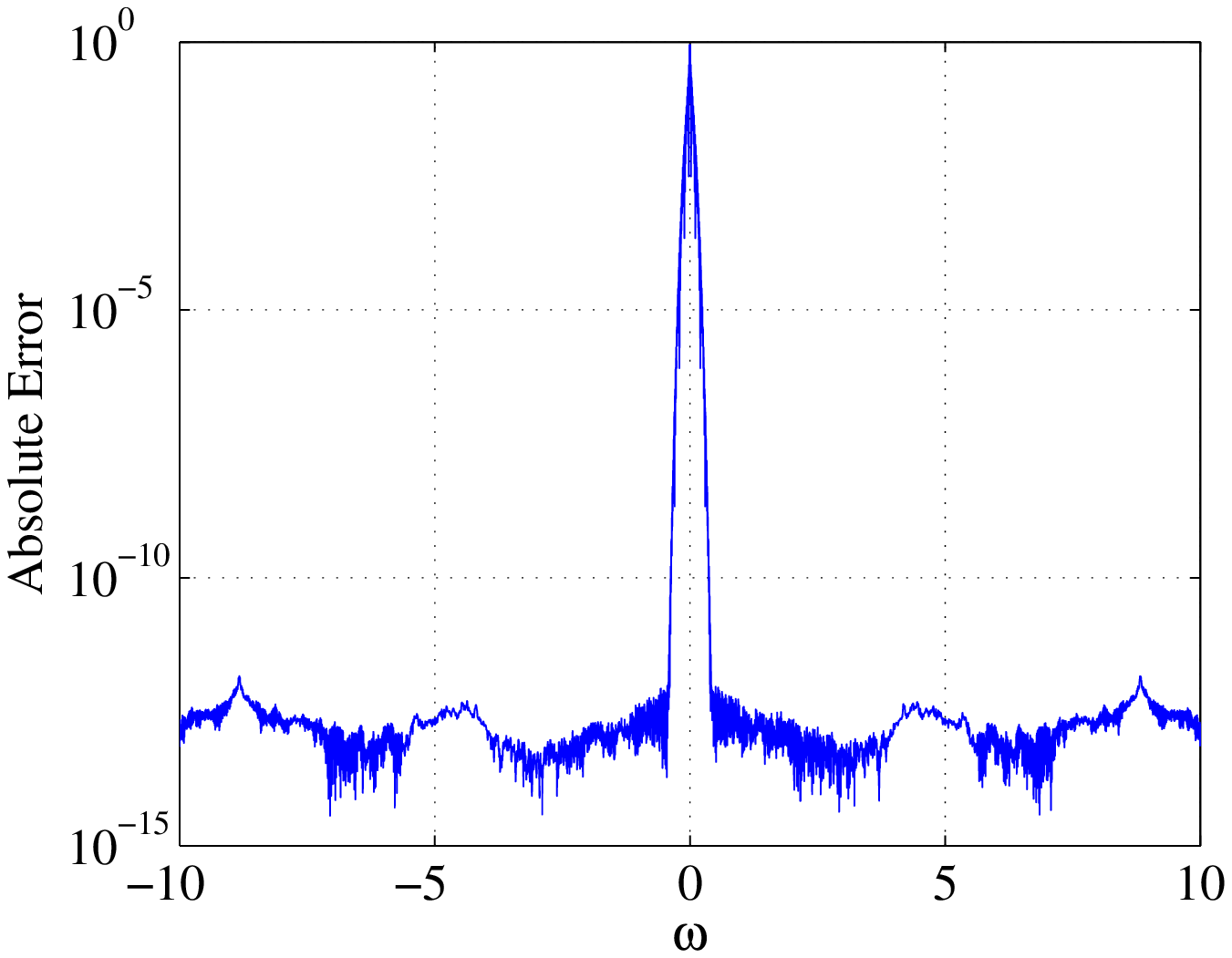}
\caption{Error for $F_{1}$ in~\eqref{eq:Ex1} 
 in the case of (B)-(b) in~\eqref{eq:RangeSet} and~\eqref{eq:ErrSet}.}
\label{fig:Ex1_01_10_m6}
\end{minipage}

\begin{minipage}{0.45\linewidth}
\includegraphics[width=.9\linewidth]{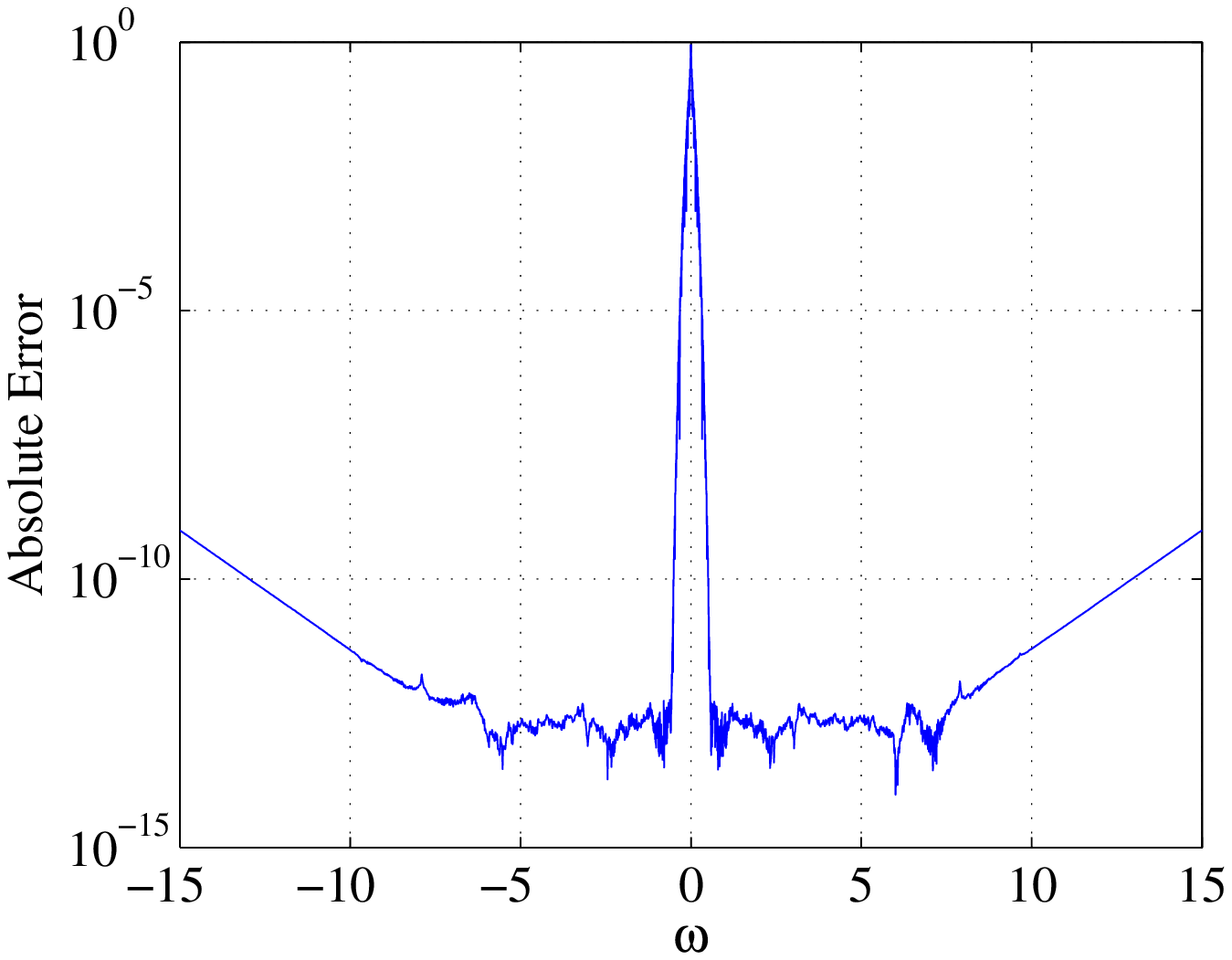}
\caption{Error for $F_{1}$ in~\eqref{eq:Ex1} 
 in the case of (C)-(a) in~\eqref{eq:RangeSet} and~\eqref{eq:ErrSet}.}
\label{fig:Ex1_01_15_m3}
\end{minipage}
\quad  
\begin{minipage}{0.45\linewidth}
\includegraphics[width=.9\linewidth]{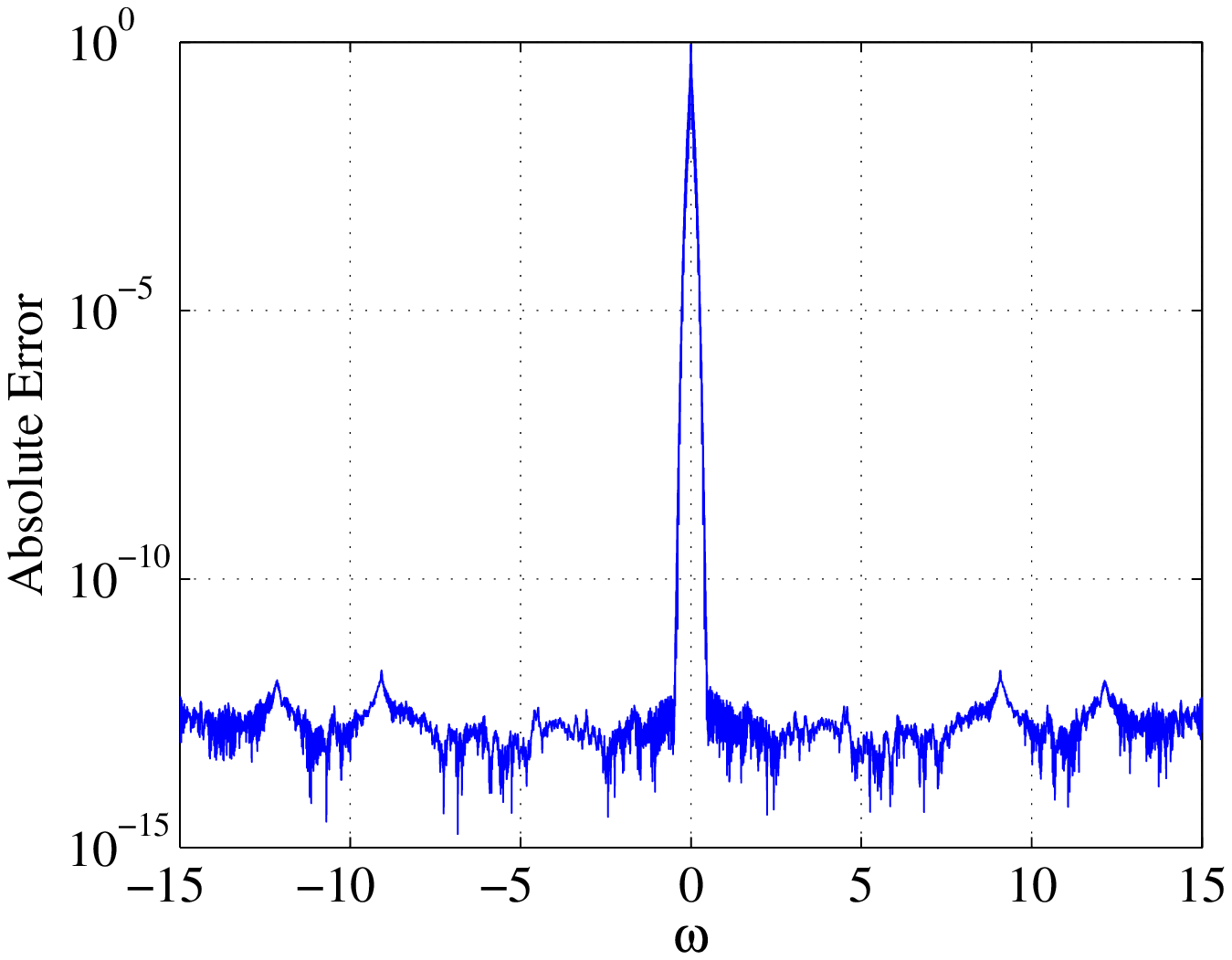}
\caption{Error for $F_{1}$ in~\eqref{eq:Ex1} 
 in the case of (C)-(b) in~\eqref{eq:RangeSet} and~\eqref{eq:ErrSet}.}
\label{fig:Ex1_01_15_m6}
\end{minipage}
\end{center}
\end{figure}

\begin{table}[t]
\begin{center}
\caption{The values of $N_{\epsilon}$ determined in the step 2 of the procedure 
and computation times for Example~\ref{ex:Ex2_Gamma}.}
\label{tab:N_e_and_Time_Ex2}
\begin{tabular}{c c  r  c }
 & & $N_{\epsilon}$ & Time (sec.) \\
\hline
\multirow{2}{*}{(A)} & (a) &   $511$ & $0.040$ \\
                         & (b) & $2047$ & $0.134$ \\
\multirow{2}{*}{(B)} & (a) & $2047$ & $0.134$ \\
                         & (b) & $8191$ & $0.507$ \\
\multirow{2}{*}{(C)} & (a) & $2047$ & $0.134$ \\
                         & (b) & $4095$ & $0.259$ \\
\end{tabular}
\end{center}
\end{table}

\begin{figure}[H]
\begin{center}
\begin{minipage}{0.45\linewidth}
\includegraphics[width=.9\linewidth]{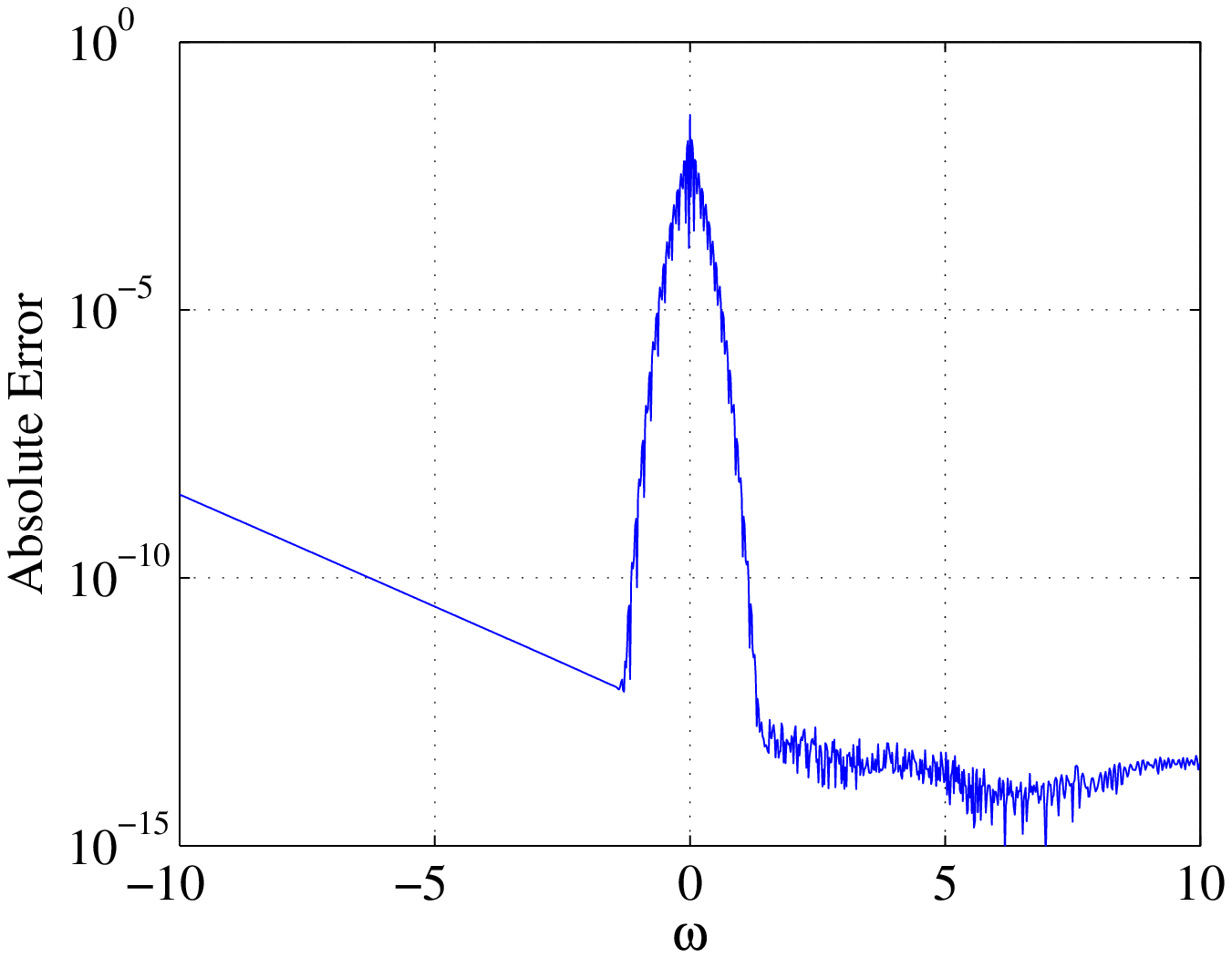}
\caption{Error for $F_{2}$ in~\eqref{eq:Ex2} 
 in the case of (A)-(a) in~\eqref{eq:RangeSet} and~\eqref{eq:ErrSet}.}
\label{fig:Ex2_02_10_m3}
\end{minipage}
\quad  
\begin{minipage}{0.45\linewidth}
\includegraphics[width=.9\linewidth]{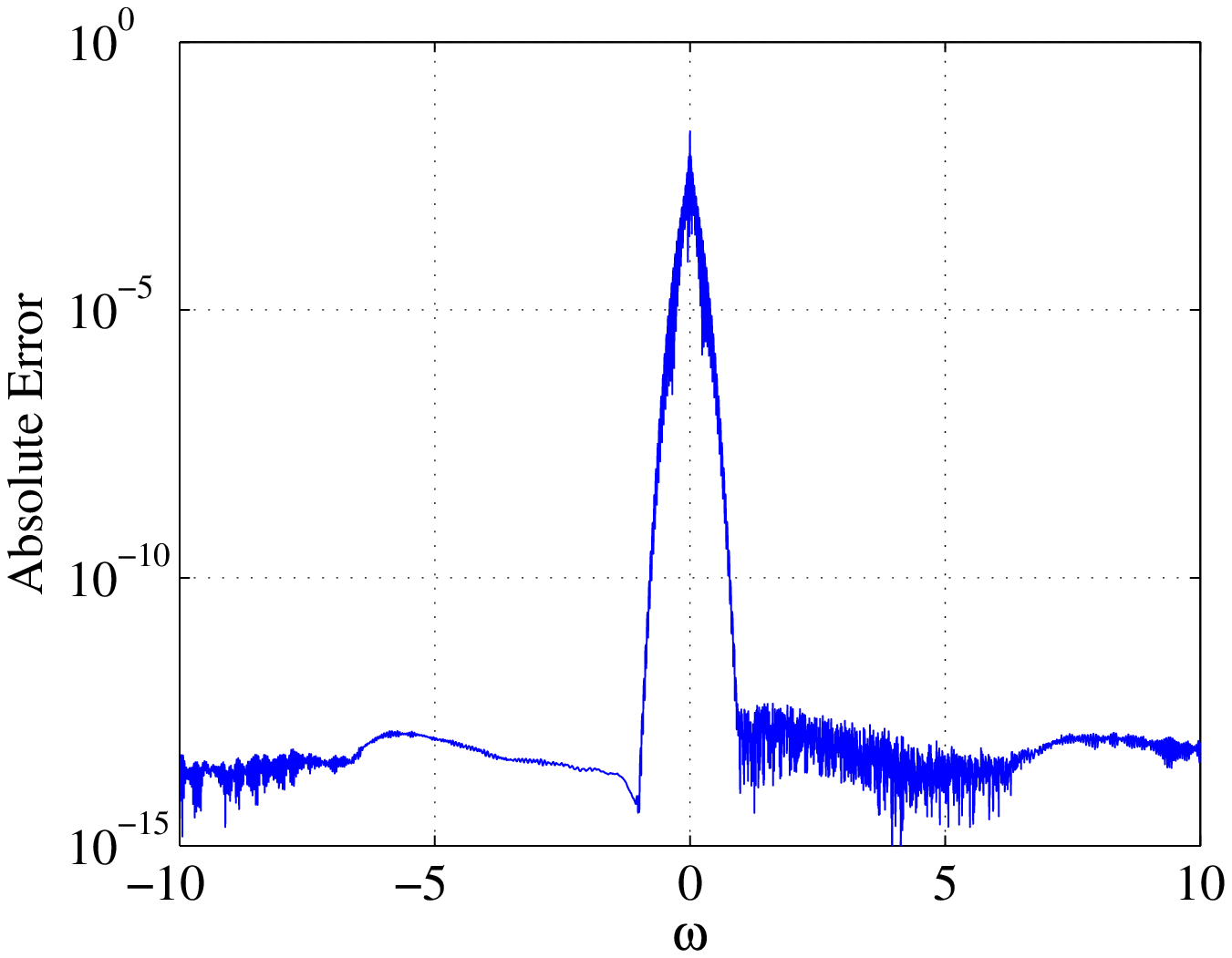}
\caption{Error for $F_{2}$ in~\eqref{eq:Ex2} 
 in the case of (A)-(b) in~\eqref{eq:RangeSet} and~\eqref{eq:ErrSet}.}
\label{fig:Ex2_02_10_m6}
\end{minipage}

\begin{minipage}{0.45\linewidth}
\includegraphics[width=.9\linewidth]{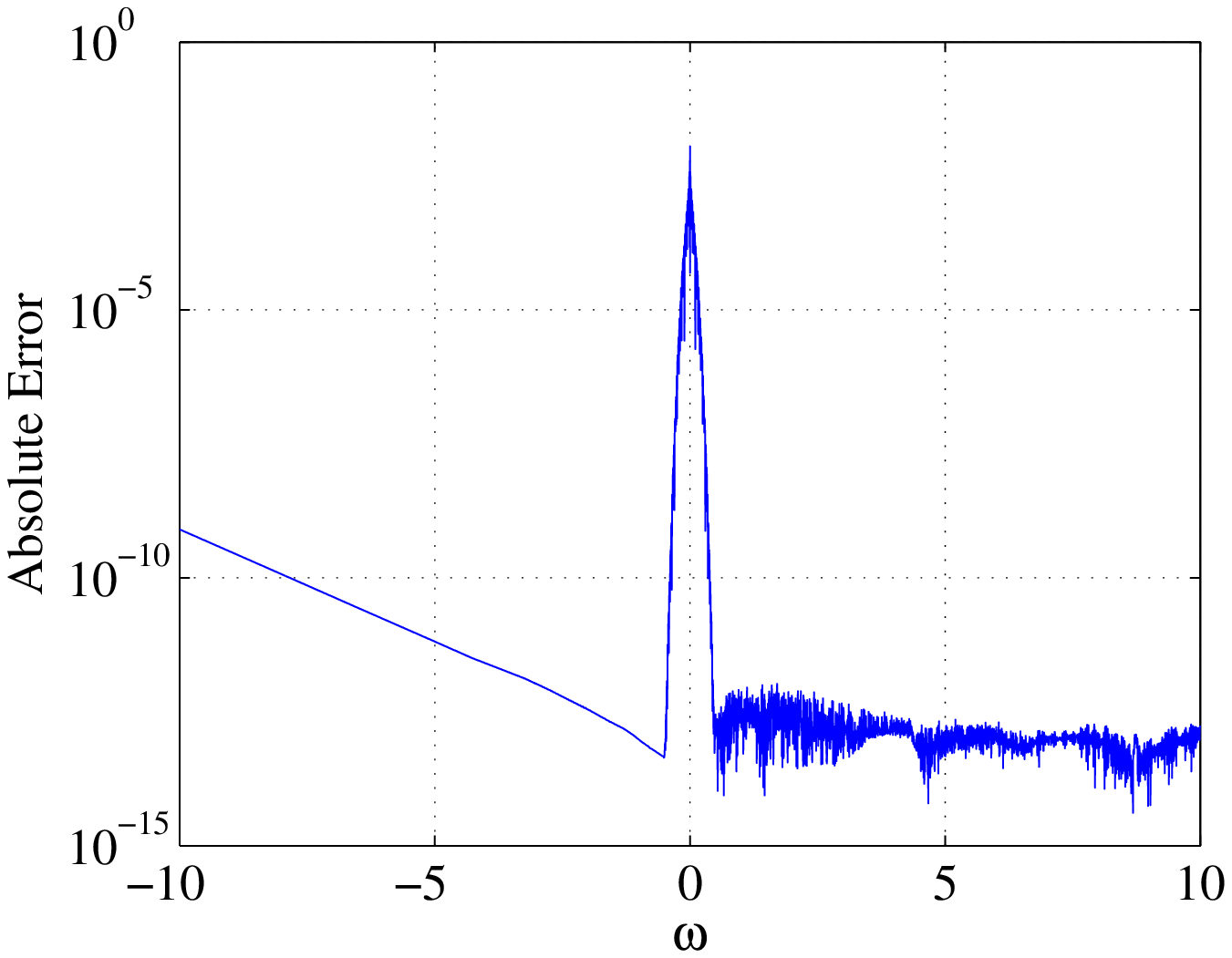}
\caption{Error for $F_{2}$ in~\eqref{eq:Ex2} 
 in the case of (B)-(a) in~\eqref{eq:RangeSet} and~\eqref{eq:ErrSet}.}
\label{fig:Ex2_01_10_m3}
\end{minipage}
\quad  
\begin{minipage}{0.45\linewidth}
\includegraphics[width=.9\linewidth]{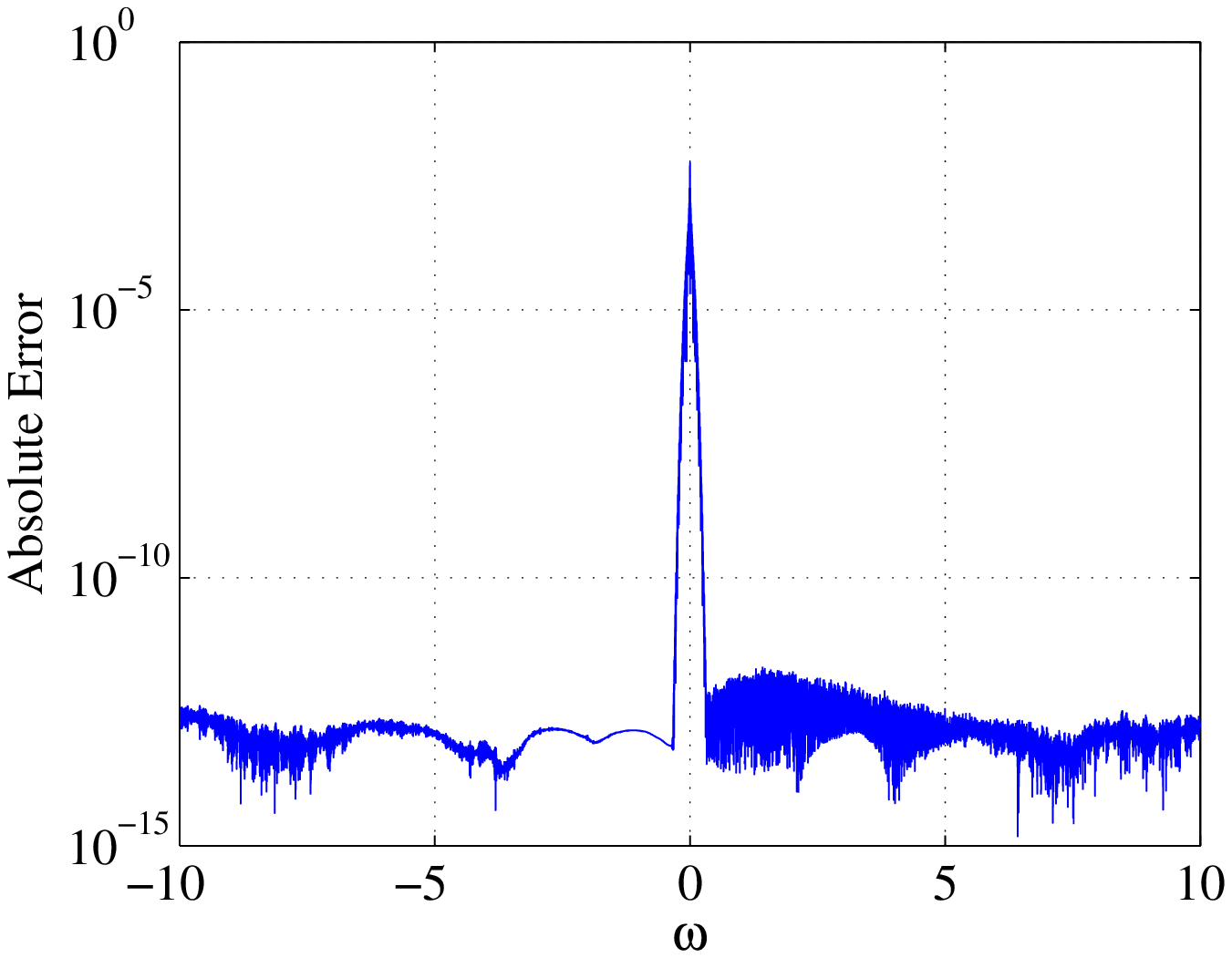}
\caption{Error for $F_{2}$ in~\eqref{eq:Ex2} 
 in the case of (B)-(b) in~\eqref{eq:RangeSet} and~\eqref{eq:ErrSet}.}
\label{fig:Ex2_01_10_m6}
\end{minipage}

\begin{minipage}{0.45\linewidth}
\includegraphics[width=.9\linewidth]{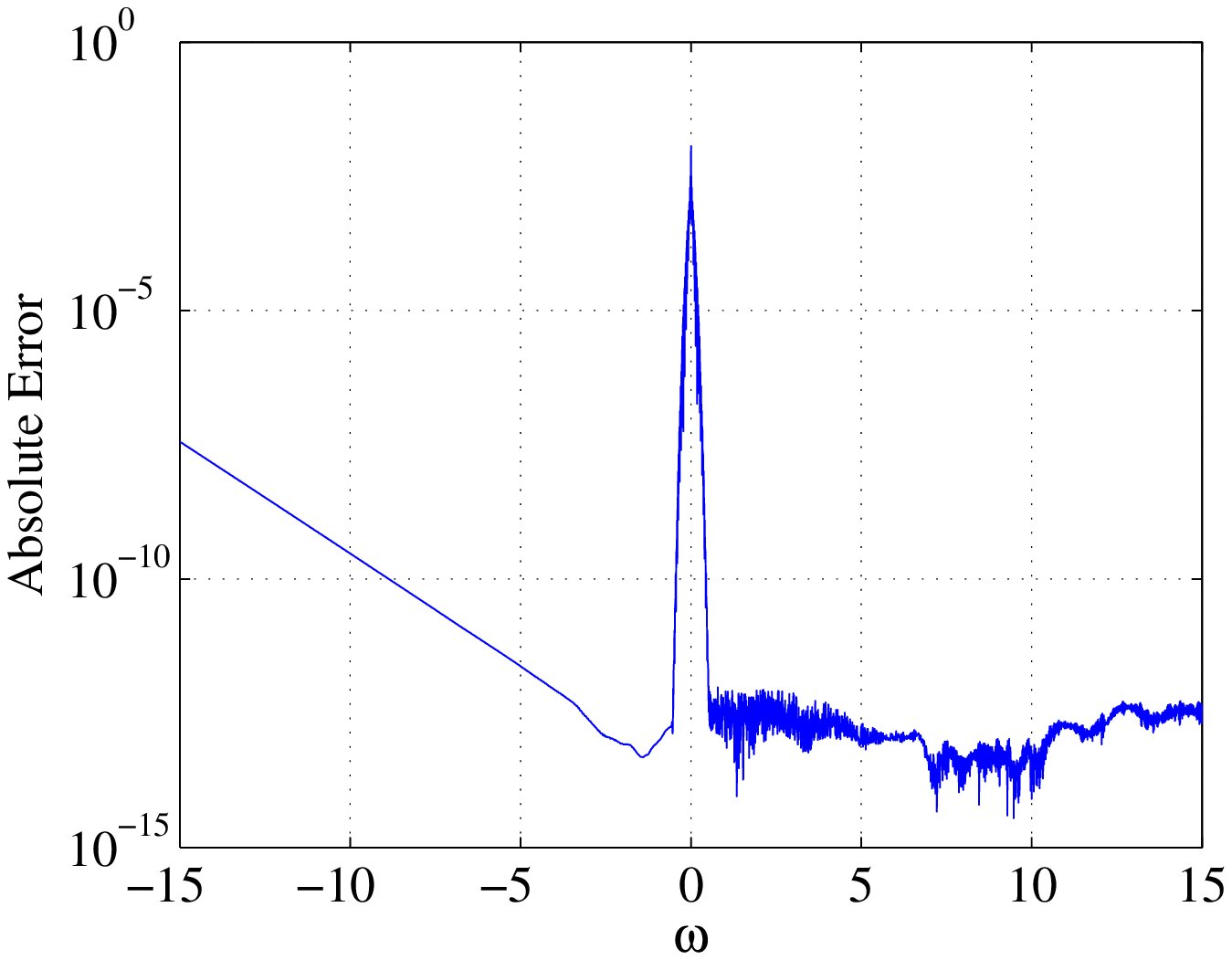}
\caption{Error for $F_{2}$ in~\eqref{eq:Ex2} 
 in the case of (C)-(a) in~\eqref{eq:RangeSet} and~\eqref{eq:ErrSet}.}
\label{fig:Ex2_01_15_m3}
\end{minipage}
\quad  
\begin{minipage}{0.45\linewidth}
\includegraphics[width=.9\linewidth]{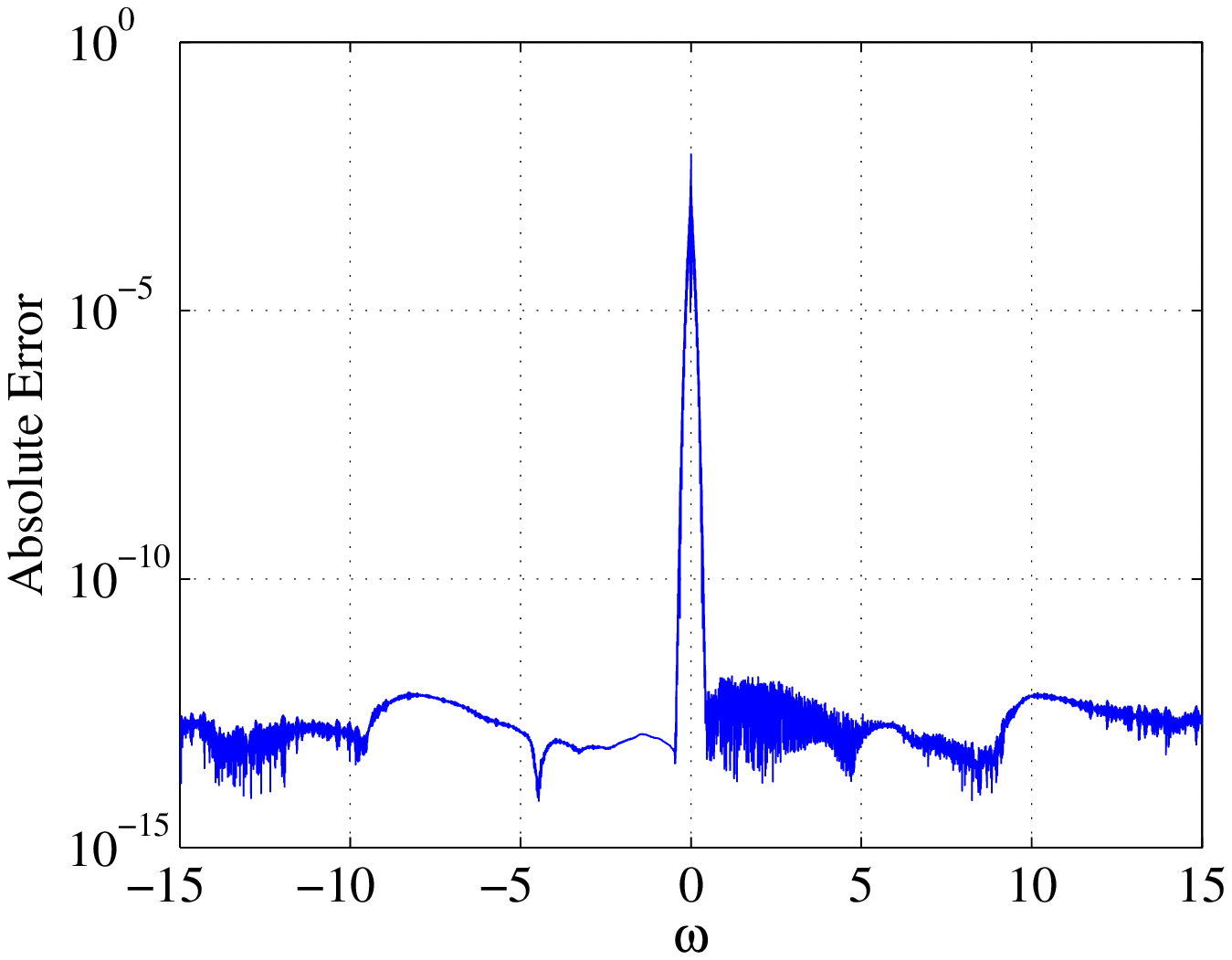}
\caption{Error for $F_{2}$ in~\eqref{eq:Ex2} 
 in the case of (C)-(b) in~\eqref{eq:RangeSet} and~\eqref{eq:ErrSet}.}
\label{fig:Ex2_01_15_m6}
\end{minipage}
\end{center}
\end{figure}

\begin{figure}[ht]
\begin{center}
\includegraphics[width=.432\linewidth]{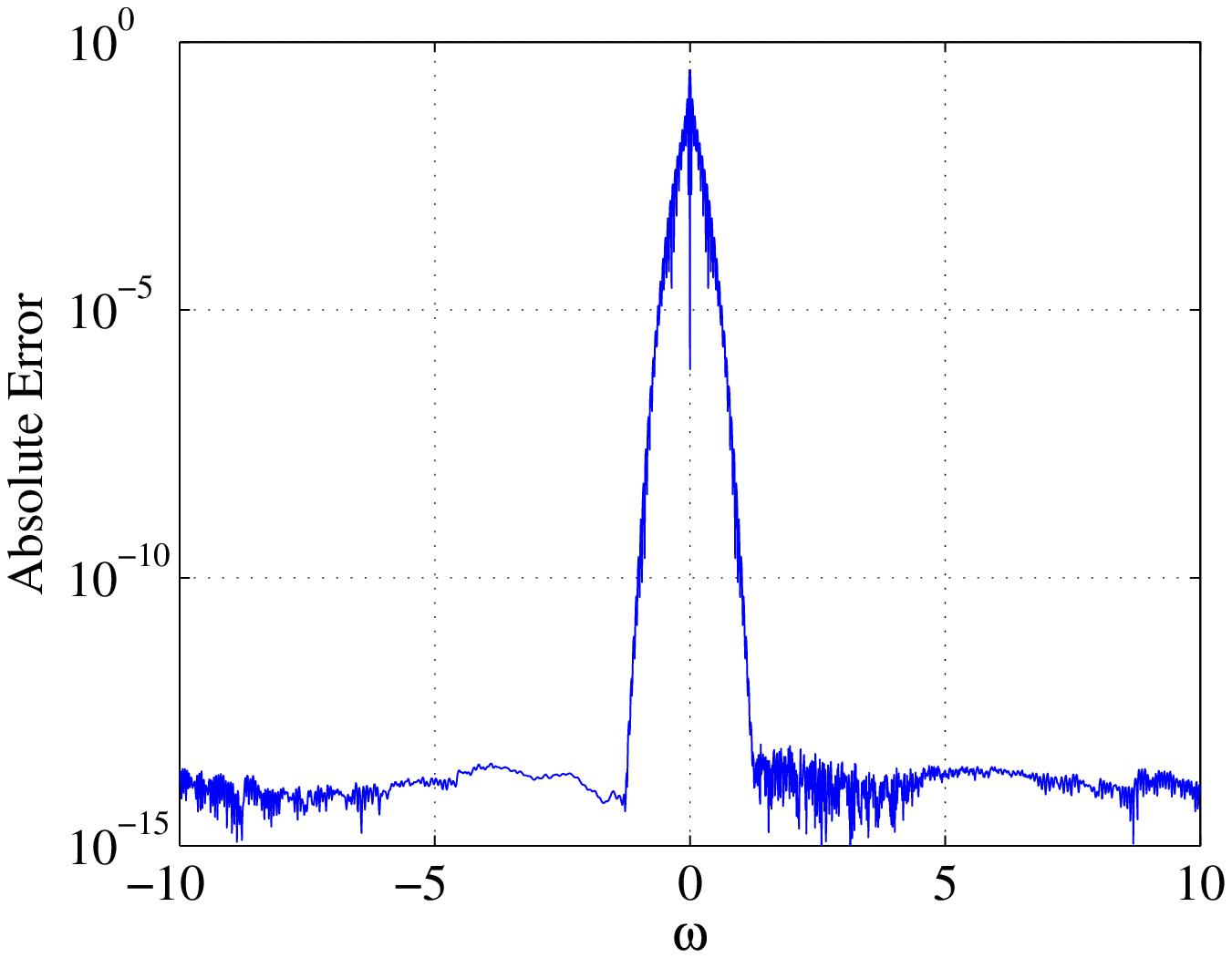}
\caption{Error for $G_{2}$ in~\eqref{eq:GaDist} 
 in the case of (A)-(a) in~\eqref{eq:RangeSet} and~\eqref{eq:ErrSet}.}
\label{fig:Ex3_02_10_m3}
\end{center}
\end{figure}

\newpage

\section{Proofs of Lemmas}
\label{sec:Proofs}

In this section, 
we present proofs of Lemmas~\ref{lem:DiscErr} and~\ref{lem:TrunErr}.
We begin with Lemma~\ref{lem:TrunErr} in Section~\ref{sec:LemTrun} because it is easier.
Then, Lemma~\ref{lem:DiscErr} is proved in Section~\ref{sec:LemDisc} and 
some other lemmas needed for Lemma~\ref{lem:DiscErr} are proved in 
Section~\ref{sec:SubLemProofs}.

\subsection{Proof of Lemma~\ref{lem:TrunErr}}
\label{sec:LemTrun}

Noting
\begin{align}
E_{w}^{(N, h)}(\omega)
& =
\left|
h \sum_{n<-N-1,\, N < n} w(|nh|; p, q) f(nh)\, \mathrm{e}^{-\mathrm{i}\, \omega\, nh}
\right|
\end{align}
and $Nh > pq$, we obtain the conclusion of Lemma~\ref{lem:TrunErr} as follows:
\begin{align}
E_{w}^{(N, h)}(\omega)
& \leq
M\, h \sum_{n > N} | w(|nh|; p, q) | 
\leq
2\, M\, h \sum_{n = N+1}^{\infty} \frac{1}{2} \exp \left\{ -\frac{(nh - pq)^{2}}{p^{2}}  \right\} \notag \\
& \leq
M\, \int_{N h}^{\infty}  \exp \left\{ -\frac{(t - pq)^{2}}{p^{2}} \right\} \mathrm{d}t 
\leq
M\, \exp \left\{ -\frac{(Nh- pq)^{2}}{p^{2}} \right\} 
\int_{0}^{\infty}  \exp \left\{ -\frac{s^{2}}{p^{2}} \right\} \mathrm{d}s \notag \\
& = 
\frac{\sqrt{\pi}\, p\, M}{2}\, \exp \left\{ -\frac{(Nh- pq)^{2}}{p^{2}} \right\}.
\end{align}

\subsection{Proof of Lemma~\ref{lem:DiscErr}}
\label{sec:LemDisc}

For $z \in \mathcal{D}_{d}$, 
let $w^{\mathrm{sym}}(z;p,q)$ be a function defined by
\begin{align}
w^{\mathrm{sym}}(z;p,q) 
=
\begin{cases}
w(z;p,q) & (\mathrm{Re}\, z \geq 0), \\
w(-z;p,q) & (\mathrm{Re}\, z < 0).
\end{cases}
\label{eq:w_sym}
\end{align}
Then, 
$F_{w}^{(\infty, h)}(\omega)$ in~\eqref{eq:F_w_infty_h} 
is rewritten in the form
\begin{align}
F_{w}^{(\infty, h)}(\omega)
& =
h \sum_{n=-\infty}^{\infty} w^{\mathrm{sym}}(nh; p, q) f(nh)\, \mathrm{e}^{-\mathrm{i}\, \omega\, nh}. 
\end{align}
To prove Lemma~\ref{lem:DiscErr}, we use the following fundamental fact.
\begin{lem}[Poisson summation formula {\cite[Theorem 1.3.1]{bib:StengerSinc2011}}]
\label{lem:PoissonSum}
Let $g \in L^{2}(\mathbf{R})$, 
and let $g$ and its Fourier transform $\hat{g}$ satisfy the conditions
\begin{align}
g(x) & = \lim_{t \to +0} \frac{g(x - t) + g(x + t)}{2}, \label{eq:PoAssump_1} \\
\hat{g}(\omega) &= \lim_{t \to +0} \frac{\hat{g}(\omega - t) + \hat{g}(\omega + t)}{2} \label{eq:PoAssump_2}
\end{align}
for $t$, $x$ and $\omega$ on $\mathbf{R}$. 
Then, for $\omega$ satisfying \eqref{eq:PoAssump_2} and $h > 0$, the following holds:
\begin{align}
h \sum_{n = -\infty}^{\infty} g(nh)\, \mathrm{e}^{-\mathrm{i}\, \omega\, n\, h}
=
\sum_{m = -\infty}^{\infty} \hat{g}\left( \frac{2\pi m}{h} - \omega \right).
\label{eq:PoissonSum}
\end{align}
\end{lem}

A function $g$ defined by $g(z) = w^{\mathrm{sym}}(z; p, q) f(z)$
belongs to $L^{2}(\mathbf{R})$ and continuous on $\mathbf{R}$, 
and its Fourier transform $\hat{g}$ is continuous on $\mathbf{R}$. 
Then, we can use Lemma~\ref{lem:PoissonSum} to yield
\begin{align}
F_{w}(\omega) - F_{w}^{(\infty, h)}(\omega)
=
\hat{g}(\omega) - 
h \sum_{n = -\infty}^{\infty} g(nh)\, \mathrm{e}^{-\mathrm{i}\, \omega\, n\, h}
=
- \sum_{m \neq 0} \hat{g}\left( \frac{2\pi m}{h} - \omega \right).
\label{eq:EstimByPoissonSum}
\end{align}
If the function $\hat{g}$ decays exponentially as 
$|\hat{g}(\omega)| = \mathrm{O}(\exp(-d\, |\omega|))\ ( \omega \to \pm \infty)$, 
the conclusion of Lemma~\ref{lem:DiscErr} may follow immediately. 
The exponential decay of $\hat{g}$, however, is not so straightforwardly shown
because $w^{\mathrm{sym}}(z; p, q)$ is not analytic on $\mathcal{D}_{d}$. 
Then, in the following, 
we apply smoothing technique to $w^{\mathrm{sym}}(z; p, q)$ 
to yield an analytic function $g_{\varepsilon}$ approximating $g$, 
and apply Lemma~\ref{lem:PoissonSum} to $g_{\varepsilon}$.


For sufficiently small $\varepsilon > 0$, 
we define $w^{\mathrm{sym}}_{\varepsilon}(z; p, q)$ by
\begin{align}
w^{\mathrm{sym}}_{\varepsilon}(z; p, q)
=
\int_{-\infty}^{\infty} 
u_{\varepsilon}(z-y)\,
w^{\mathrm{sym}}(y; p, q)\, \mathrm{d}y,
\label{eq:w_sym_eps}
\end{align}
where
\begin{align}
u_{\varepsilon}(z) 
= 
\frac{1}{\sqrt{2\pi}\, \varepsilon} \exp\left( - \frac{z^{2}}{2\, \varepsilon^{2}} \right), 
\label{eq:u_e}
\end{align}
and consider $g_{\varepsilon}(z) = w^{\mathrm{sym}}_{\varepsilon}(z; p, q) f(z)$.
Using $w^{\mathrm{sym}}_{\varepsilon}(z; p, q)$, 
we estimate $E_{w}^{(\infty, h)}(\omega)$
 as follows:
\begin{align}
E_{w}^{(\infty, h)}(\omega) 
= | F_{w}(\omega) - F_{w}^{(\infty, h)}(\omega) | 
\leq 
E_{w, \varepsilon, \mathrm{int}}^{(\infty, h)}(\omega)
+ E_{w, \varepsilon, \mathrm{sum}}^{(\infty, h)}(\omega)
+ E_{w, \varepsilon}^{(\infty, h)}(\omega),
\label{eq:Div_E_w}
\end{align}
where
\begin{align}
E_{w, \varepsilon, \mathrm{int}}^{(\infty, h)}(\omega)
& = 
\left|
\int_{-\infty}^{\infty}
(w^{\mathrm{sym}}_{\varepsilon}(x; p, q) - w^{\mathrm{sym}}(x; p, q))\, 
f(x)
\, \mathrm{e}^{-\mathrm{i}\, \omega\, x}
\, \mathrm{d}x
\right|, \\
E_{w, \varepsilon, \mathrm{sum}}^{(\infty, h)}(\omega)
& = 
\left|
h \sum_{n=-\infty}^{\infty} 
(w^{\mathrm{sym}}_{\varepsilon}(nh; p, q) - w^{\mathrm{sym}}(nh; p, q))\, 
f(nh)
\, \mathrm{e}^{-\mathrm{i}\, \omega\, nh}
\right|, \\
E_{w, \varepsilon}^{(\infty, h)}(\omega)
& = 
\left|
\int_{-\infty}^{\infty} w^{\mathrm{sym}}_{\varepsilon}(x; p, q)\, f(x)\, \mathrm{d}x
-
h \sum_{n=-\infty}^{\infty} 
w^{\mathrm{sym}}_{\varepsilon}(nh; p, q) f(nh)\, \mathrm{e}^{-\mathrm{i}\, \omega\, nh}
\right|.
\end{align}
To estimate these, we use the following three lemmas, 
whose proofs are shown in Section~\ref{sec:SubLemProofs} later.

\begin{lem}
\label{lem:w_e_analytic}
The function $w^{\mathrm{sym}}_{\varepsilon}(z; p, q)$ is analytic on $\mathbf{C}$. 
\end{lem}

\begin{lem}
\label{lem:w_e_int_zero}
The difference 
$w^{\mathrm{sym}}_{\varepsilon}(z; p, q) - w^{\mathrm{sym}}(z; p, q)$
is absolutely integrable on $\{ t \pm d\, \mathrm{i} \mid t \in \mathbf{R} \}$ and 
\begin{align}
\lim_{\varepsilon \to +0} 
\int_{-\infty}^{\infty}
\left| 
w^{\mathrm{sym}}_{\varepsilon}(t \pm d\, \mathrm{i}; p, q) 
- w^{\mathrm{sym}}(t \pm d\, \mathrm{i}; p, q) 
\right|\, 
\mathrm{d}t = 0.
\end{align}
\end{lem}

\begin{lem}
\label{lem:w_e_sum_zero}
For any $h>0$, the following holds:
\begin{align}
\lim_{\varepsilon \to +0} 
h
\sum_{n = -\infty}^{\infty}
\left| 
w^{\mathrm{sym}}_{\varepsilon}(n h ; p, q) 
- w^{\mathrm{sym}}(n h ; p, q) 
\right|
 = 0.
\end{align}
\end{lem}

Now we are in position to prove Lemma~\ref{lem:DiscErr}.

\begin{proof}[Proof of Lemma~\ref{lem:DiscErr}]
As for $E_{w, \varepsilon, \mathrm{int}}^{(\infty, h)}(\omega)$
and $E_{w, \varepsilon, \mathrm{sum}}^{(\infty, h)}(\omega)$, 
we note that
\begin{align}
E_{w, \varepsilon, \mathrm{int}}^{(\infty, h)}(\omega)
& \leq 
M \int_{-\infty}^{\infty}
\left| 
w^{\mathrm{sym}}_{\varepsilon}(x; p, q) 
- w^{\mathrm{sym}}(x; p, q) 
\right|\, 
\mathrm{d}x, \\
E_{w, \varepsilon, \mathrm{sum}}^{(\infty, h)}(\omega)
& \leq 
M\, h \sum_{n = -\infty}^{\infty}
\left| 
w^{\mathrm{sym}}_{\varepsilon}(n h ; p, q) 
- w^{\mathrm{sym}}(n h ; p, q) 
\right|. 
\end{align}
From these and Lemmas~\ref{lem:w_e_int_zero} and~\ref{lem:w_e_sum_zero}, 
we can make $E_{w, \varepsilon, \mathrm{int}}^{(\infty, h)}(\omega)$
and $E_{w, \varepsilon, \mathrm{sum}}^{(\infty, h)}(\omega)$ arbitrarily small
by choosing~$\varepsilon$ independently of~$\omega$. 

As for $E_{w, \varepsilon}^{(\infty, h)}(\omega)$, 
we use Lemmas~\ref{lem:PoissonSum}, \ref{lem:w_e_analytic} and~\ref{lem:w_e_int_zero}.
Instead of~\eqref{eq:EstimByPoissonSum},  
for $g_{\varepsilon}(z) = w^{\mathrm{sym}}_{\varepsilon}(z; p, q) f(z)$, we use 
\begin{align}
\hat{g}_{\varepsilon}(\omega) - 
h \sum_{n = -\infty}^{\infty} g_{\varepsilon}(nh)\, \mathrm{e}^{-\mathrm{i}\, \omega\, n\, h}
=
- \sum_{m \neq 0} \hat{g}_{\varepsilon}\left( \frac{2\pi m}{h} - \omega \right)
\label{eq:EstimByPoissonSum_eps}
\end{align}
for the estimate. 
For $\omega \geq 0$, 
Lemma~\ref{lem:w_e_analytic} allows us 
to move the path of the integral defining $\hat{g}_{\varepsilon}$ as follows:
\begin{align}
\hat{g}_{\varepsilon}(\omega)
= 
\int_{-\infty}^{\infty}
g_{\varepsilon}(x)\, \mathrm{e}^{-\mathrm{i}\, \omega\, x}\, \mathrm{d}x
=
\int_{-\infty}^{\infty}
g_{\varepsilon}(t - d\, \mathrm{i})\, \mathrm{e}^{-\mathrm{i}\, \omega\, (t - d\, \mathrm{i}) }\, \mathrm{d}t
=
\mathrm{e}^{-\omega\, d}
\int_{-\infty}^{\infty}
g_{\varepsilon}(t - d\, \mathrm{i})\, \mathrm{e}^{-\mathrm{i}\, \omega\, t}\, \mathrm{d}t.
\end{align}
Since a similar argument is possible for $\omega < 0$, 
we have
\begin{align}
| \hat{g}_{\varepsilon}(\omega) |
\leq
\mathrm{e}^{-|\omega|\, d}
\max_{s \in \{-1, +1\}}
\left|
\int_{-\infty}^{\infty}
g_{\varepsilon}(t + s\, d\, \mathrm{i})\, \mathrm{e}^{-\mathrm{i}\, \omega\, t}\, \mathrm{d}t
\right|. 
\label{eq:gEstInt}
\end{align}
As for the integral in~\eqref{eq:gEstInt}, we have 
\begin{align}
& \int_{-\infty}^{\infty}
| g_{\varepsilon}(t \pm d\, \mathrm{i}) | \, \mathrm{d}t \notag \\
& \leq M
\left(
\int_{-\infty}^{\infty}
| w^{\mathrm{sym}}_{\varepsilon}(t \pm d\, \mathrm{i}; p, q)
- w^{\mathrm{sym}}(t \pm d\, \mathrm{i}; p, q) |
\, \mathrm{d}t
+
\int_{-\infty}^{\infty}
| w^{\mathrm{sym}}(t \pm d\, \mathrm{i}; p, q) | \, \mathrm{d}t
\right),
\label{eq:int_g_1}
\end{align}
and 
\begin{align}
\int_{-\infty}^{\infty}
| w^{\mathrm{sym}}(t \pm d\, \mathrm{i}; p, q) | \, \mathrm{d}t
\leq
\left(
\frac{\sqrt{\pi}}{2} + q
\right)\, p\, \exp\{ (d / p)^{2} \}. 
\label{eq:int_g_2}
\end{align}
Then, due to~\eqref{eq:gEstInt}, \eqref{eq:int_g_1}, \eqref{eq:int_g_2} and Lemma~\ref{lem:w_e_int_zero}, 
we have
\begin{align}
| \hat{g}_{\varepsilon}(\omega) |
\leq
M
\left( 
\delta(\varepsilon) 
+
\left(
\frac{\sqrt{\pi}}{2} + q
\right)\, p\, \exp\{ (d / p)^{2} \}
\right)
\exp(-d\, | \omega |), 
\label{eq:g_hat_decay}
\end{align}
where $\lim_{\varepsilon \to +0} \delta(\varepsilon) = 0$. 
Finally, 
using~\eqref{eq:EstimByPoissonSum_eps} and~\eqref{eq:g_hat_decay}, 
for $\omega$ with $| \omega | \leq \pi / h$,
we have
\begin{align}
E_{w, \varepsilon}^{(\infty, h)}(\omega)
& \leq 
\sum_{m \neq 0} \left| \hat{g}_{\varepsilon}\left( \frac{2\pi m}{h} - \omega \right) \right|
\leq 
2\, C_{M, \varepsilon, p, q, d}
\sum_{m = 1}^{\infty} \exp \left( - \frac{\pi\, d}{h} ( 2 |m| - 1 ) \right) \\
& =
\frac{2\, C_{M, \varepsilon, p, q, d}}{1 - \exp(- 2\, \pi\, d/h)}
\exp \left( - \frac{\pi\, d}{h} \right), 
\end{align}
where
\begin{align}
C_{M, \varepsilon, p, q, d} 
=
M
\left( 
\delta(\varepsilon) 
+
\left(
\frac{\sqrt{\pi}}{2} + q
\right)\, p\, \exp\{ (d / p)^{2} \}
\right).  
\end{align}
Since $\varepsilon > 0$ is arbitrary, 
the conclusion of Lemma~\ref{lem:DiscErr} follows from~\eqref{eq:Div_E_w} 
and the above estimates.
\end{proof}

\subsection{Proofs of Lemmas~\ref{lem:w_e_analytic}--\ref{lem:w_e_sum_zero}}
\label{sec:SubLemProofs}

\begin{proof}[Proof of Lemma~\ref{lem:w_e_analytic}]
Let $\hat{w}^{\mathrm{sym}}_{\varepsilon}(\omega; p, q)$ 
be the Fourier transform of $w^{\mathrm{sym}}_{\varepsilon}(z; p, q)$.
Then, $\hat{w}^{\mathrm{sym}}_{\varepsilon}(\omega; p, q)$ is written in the form
\begin{align}
\hat{w}^{\mathrm{sym}}_{\varepsilon}(\omega; p, q)
=
\exp\left( - \frac{\varepsilon^{2}\, \omega^{2} }{2} \right)\, 
\hat{w}^{\mathrm{sym}}(\omega; p, q),
\end{align}
where $\hat{w}^{\mathrm{sym}}(\omega; p, q)$ is 
the Fourier transform of $w^{\mathrm{sym}}(z; p, q)$.
Since $\hat{w}^{\mathrm{sym}}(\omega; p, q)$ 
is uniformly bounded for any $\omega \in \mathbf{R}$, 
it follows that
$\hat{w}^{\mathrm{sym}}_{\varepsilon}(\omega; p, q) 
= \mathrm{O}(\exp(-(\varepsilon^{2} / 2)\, \omega^{2}))\ (\omega \to \pm \infty)$. 
Then, the integral
\begin{align}
\int_{-\infty}^{\infty}
-\mathrm{i}\, \omega\, 
\hat{w}^{\mathrm{sym}}_{\varepsilon}(\omega; p, q)\, 
\mathrm{e}^{-\mathrm{i}\, \omega\, z}\, \mathrm{d}\omega
\end{align}
converges absolutely for any $z \in \mathbf{C}$, 
which implies $w^{\mathrm{sym}}_{\varepsilon}(z; p, q)$ is differentiable for any $z \in \mathbf{C}$.
\end{proof}

\begin{proof}[Proof of Lemma~\ref{lem:w_e_int_zero}]
By substitution of $y = z - \varepsilon\, v$ into~\eqref{eq:w_sym_eps}, 
$w^{\mathrm{sym}}_{\varepsilon}(z; p, q)$ is rewritten in the form
\begin{align}
w^{\mathrm{sym}}_{\varepsilon}(z; p, q)
=
\int_{-\infty}^{\infty} 
u_{1}(v)\,
w^{\mathrm{sym}}(z - \varepsilon\, v; p, q)\, \mathrm{d}v,
\end{align}
where $u_{\varepsilon}$ is defined by~\eqref{eq:u_e}.
Then, we have
\begin{align}
\left| 
w^{\mathrm{sym}}_{\varepsilon}(z; p, q) - w^{\mathrm{sym}}(z; p, q)
\right|
& =
\left| 
\int_{-\infty}^{\infty} 
u_{1}(v)\,
\left( 
w^{\mathrm{sym}}(z - \varepsilon\, v; p, q) - 
w^{\mathrm{sym}}(z; p, q)
\right)
\mathrm{d}v
\right| \notag \\
& \leq 
\int_{-\infty}^{\infty} 
u_{1}(v)\,
\left|
w^{\mathrm{sym}}(z - \varepsilon\, v; p, q) - 
w^{\mathrm{sym}}(z; p, q)
\right|
\mathrm{d}v,
\label{eq:int_w_estim_key_pre}
\end{align}
and therefore
\begin{align}
& \int_{-\infty}^{\infty}
\left| 
w^{\mathrm{sym}}_{\varepsilon}(t \pm d'\, \mathrm{i}; p, q) 
- w^{\mathrm{sym}}(t \pm d'\, \mathrm{i}; p, q) 
\right|\, 
\mathrm{d}t \notag \\
& \leq 
\int_{-\infty}^{\infty} 
u_{1}(v)
\left(
\int_{-\infty}^{\infty}
\left| 
w^{\mathrm{sym}}(t - \varepsilon\, v \pm d'\, \mathrm{i}; p, q) 
- w^{\mathrm{sym}}(t \pm d'\, \mathrm{i}; p, q) 
\right|\, 
\mathrm{d}t
\right)\, \mathrm{d}v
\label{eq:int_w_estim_key}
\end{align}
for any $d'$ with $0\leq d' \leq d$. 
In the following, we estimate 
$W_{\varepsilon, v, d'}(t) = 
| w^{\mathrm{sym}}(t - \varepsilon\, v \pm d'\, \mathrm{i}; p, q) 
- w^{\mathrm{sym}}(t \pm d'\, \mathrm{i}; p, q) |$.
Due to the symmetry with respect to $v$, 
it suffices to consider the case $v\geq 0$. 
Here we note that
\begin{align}
w^{\mathrm{sym}}(z; p, q) 
& = 
\frac{\mathrm{e}^{(\mathrm{Im}\, z/p)^{2}}}{\sqrt{\pi}} 
\int_{a\, \mathrm{Re}\, z /p - q }^{\infty}
\exp\left\{ - s^{2} - 2\, \mathrm{i}\, a\, (\mathrm{Im}\, z/p) s \right\}\, \mathrm{d}s, 
\label{eq:w_estim_key}
\end{align}
where $a = 1$ if $\mathrm{Re}\, z \geq 0$ and $a = -1$ if $\mathrm{Re}\, z < 0$.
First, let $t \geq 0$.
For $t$ with $0 \leq t\leq \varepsilon\, v$, 
we have
\begin{align}
W_{\varepsilon, v, d'}(t)
& \leq 
\left| 
w^{\mathrm{sym}}(t - \varepsilon\, v \pm d'\, \mathrm{i}; p, q) 
\right|
+
\left| 
w^{\mathrm{sym}}(t \pm d'\, \mathrm{i}; p, q) 
\right| \notag \\
& \leq
\frac{\mathrm{e}^{(d'/p)^{2}}}{\sqrt{\pi}} 
\left(
\int_{(\varepsilon\, v - t) /p - q }^{\infty}
\mathrm{e}^{- s^{2}}
\, \mathrm{d}s
+
\int_{ t /p - q }^{\infty}
\mathrm{e}^{- s^{2}}
\, \mathrm{d}s
\right)
\leq 
2\, \mathrm{e}^{(d'/p)^{2}}. 
\label{eq:w_estim_aug1}
\end{align}
For $t$ with $t> \varepsilon\, v$, we have
\begin{align}
W_{\varepsilon, v, d'}(t)
\leq 
\frac{\mathrm{e}^{(d'/p)^{2}}}{\sqrt{\pi}} 
\int_{ t /p - q - \varepsilon\, v /p}^{t /p - q}
\mathrm{e}^{- s^{2}}
\, \mathrm{d}s 
\leq 
\frac{\mathrm{e}^{(d'/p)^{2}}}{\sqrt{\pi}\, p}\, \varepsilon\, v \cdot
\begin{cases}
\mathrm{e}^{- (t/p - q)^{2}} & (t < pq), \\
1 & (pq \leq t < pq + \varepsilon\, v), \\
\mathrm{e}^{- \{ t/p - (q + \varepsilon\, v / p)\}^{2}} & (pq + \varepsilon\, v \leq t ).
\end{cases} 
\label{eq:w_estim_aug2}
\end{align}
Next, for $t < 0$, we can obtain a similar estimate to \eqref{eq:w_estim_aug2} as follows:
\begin{align}
W_{\varepsilon, v, d'}(t) 
\leq 
\frac{\mathrm{e}^{(d'/p)^{2}}}{\sqrt{\pi}\, p}\, \varepsilon\, v \cdot
\begin{cases}
\mathrm{e}^{- (-t/p - q)^{2}} & (t < -pq), \\
1 & (-pq \leq t < -pq + \varepsilon\, v), \\
\mathrm{e}^{- \{ -t/p - (q - \varepsilon\, v / p)\}^{2}} & (-pq + \varepsilon\, v \leq t ).
\end{cases}
\label{eq:w_estim_aug3}
\end{align}
Combining \eqref{eq:w_estim_aug1}, \eqref{eq:w_estim_aug2} and~\eqref{eq:w_estim_aug3}, 
we have
\begin{align}
\int_{-\infty}^{\infty}
W_{\varepsilon, v, d'}(t)\,  
\mathrm{d}t 
\leq 
\mathrm{e}^{(d'/p)^{2}}\, \varepsilon\, v 
\left\{
2 + 
2 \left(
1 + \frac{\varepsilon\, v }{\sqrt{\pi}\, p} + 1
\right)
\right\}
=
\mathrm{e}^{(d'/p)^{2}}\, \varepsilon\, v 
\left(
6 + 
\frac{2\, \varepsilon\, v }{\sqrt{\pi}\, p}
\right).
\label{eq:w_estim_aug4}
\end{align}
Finally, using~\eqref{eq:int_w_estim_key} and~\eqref{eq:w_estim_aug4}
for $d' = d$, we obtain
\begin{align}
& \int_{-\infty}^{\infty}
\left| 
w^{\mathrm{sym}}_{\varepsilon}(t \pm d\, \mathrm{i}; p, q) 
- w^{\mathrm{sym}}(t \pm d\, \mathrm{i}; p, q) 
\right|\, 
\mathrm{d}t \notag \\
& \leq 
\mathrm{e}^{(d/p)^{2}}
\left(
6\, \varepsilon
\int_{-\infty}^{\infty} \, |v|\, u_{1}(v)\, \mathrm{d}v
+ 
\frac{2}{\sqrt{\pi}\, p}\, \varepsilon^{2}\, 
\int_{-\infty}^{\infty} \, |v|^{2}\, u_{1}(v)\, \mathrm{d}v
\right)
\to 0 \quad (\varepsilon \to +0), \notag
\end{align}
the conclusion of Lemma~\ref{lem:w_e_int_zero}.
\end{proof}

\begin{proof}[Proof of Lemma~\ref{lem:w_e_sum_zero}]
We use the estimates~\eqref{eq:w_estim_aug2} and~\eqref{eq:w_estim_aug3} for $d' = 0$
in the proof of Lemma~\ref{lem:w_e_int_zero}.
In addition, we modify~\eqref{eq:w_estim_aug1} 
for $t$ with $0 \leq t\leq \varepsilon\, v$ and $d' = 0$ as follows: 
\begin{align}
& \left| 
w^{\mathrm{sym}}(t - \varepsilon\, v; p, q) 
- w^{\mathrm{sym}}(t ; p, q) 
\right| 
=
\frac{1}{\sqrt{\pi}} 
\left|
\int_{(\varepsilon\, v - t) /p - q }^{t /p - q}
\exp\left( - s^{2} \right)\, \mathrm{d}s
\right|
\leq 
\frac{\varepsilon\, v}{\sqrt{\pi}\, p}.
\label{eq:w_estim_aug1_real}
\end{align}
Combining the 
estimates
\eqref{eq:w_estim_aug2}, \eqref{eq:w_estim_aug3} and \eqref{eq:w_estim_aug1_real}, 
for $v \geq 0$ and $t\in \mathbf{R}$, we have
\begin{align}
& \left| 
w^{\mathrm{sym}}(t - \varepsilon\, v; p, q) 
- w^{\mathrm{sym}}(t ; p, q) 
\right| 
\leq 
\frac{\varepsilon\, v}{\sqrt{\pi}\, p} \cdot 
\begin{cases}
\mathrm{e}^{- (-t/p - q)^{2}} & (t < -pq), \\
1 & (-pq \leq t < pq + \varepsilon\, v), \\
\mathrm{e}^{- \{ t/p - (q + \varepsilon\, v / p)\}^{2}} & (pq + \varepsilon\, v \leq t ),
\end{cases}
\end{align}
and therefore
\begin{align}
& h \sum_{n = -\infty}^{\infty} 
\left| 
w^{\mathrm{sym}}(nh - \varepsilon\, v; p, q) 
- w^{\mathrm{sym}}(nh ; p, q) 
\right| \notag \\
& \leq 
\frac{\varepsilon\, v}{\sqrt{\pi}\, p}\, h
\left(
\sum_{n \leq -1-pq/h} \mathrm{e}^{- (-nh/p - q)^{2}}
+
\sum_{
\begin{subarray}{c}
-1-pq/h < n \\
< 1 + (pq + \varepsilon\, v)/h
\end{subarray}
} 1
+ 
\sum_{1 + (pq + \varepsilon\, v)/h \leq n}
\mathrm{e}^{- \{ nh/p - (q + \varepsilon\, v / p)\}^{2}}
\right) \notag \\
& \leq 
\frac{\varepsilon\, v}{\sqrt{\pi}\, p}\, 
\left\{
\int_{-\infty}^{\infty} \mathrm{e}^{-(s/p)^{2}}\, \mathrm{d}s
+
(2h + 2pq + \varepsilon\, v)
\right\}
=
\frac{\varepsilon\, v}{\sqrt{\pi}\, p}\, 
\left\{
\sqrt{\pi}\, p
+
(2h + 2pq + \varepsilon\, v)
\right\}.
\label{eq:sum_estim_pre}
\end{align}
Then, using~\eqref{eq:int_w_estim_key_pre}, \eqref{eq:sum_estim_pre}
and the symmetry with respect to $v$, 
we have
\begin{align}
h \sum_{n = -\infty}^{\infty}
\left| 
w^{\mathrm{sym}}_{\varepsilon}(n h ; p, q) 
- w^{\mathrm{sym}}(n h ; p, q) 
\right|
& \leq 
\varepsilon
\int_{-\infty}^{\infty} 
\frac{|v|\, u_{1}(v)}{\sqrt{\pi}\, p}\, 
\left\{
\sqrt{\pi}\, p
+
(2h + 2pq + \varepsilon\, |v|)
\right\} 
\to 0 \ (\varepsilon \to +0), 
\end{align}
the conclusion of Lemma~\ref{lem:w_e_sum_zero}.
\end{proof}

\section{Concluding Remarks}
\label{sec:Concl}

In this paper, 
we considered error control of the formula~\eqref{eq:TargetComputation} with 
continuous Euler transform $w$ for the Fourier transform $F(\omega)$ of slowly decaying functions.
Based on the error estimate of the formula shown by Lemmas~\ref{lem:E_w}--\ref{lem:TrunErr}, 
we presented Theorem~\ref{thm:FormulaParaErr}
showing appropriate setting of the parameters in the formula 
to compute approximate values of the Fourier transform 
with given accuracy in given ranges of the frequency $\omega$.
Furthermore, 
combining the formula and the fractional FFT, 
we showed that the computation can be done 
in the same order of computation time as the FFT.
Improvement of the estimate of Theorem~\ref{thm:FormulaParaErr}
and application of the formula to parabolic partial differential equations, etc. 
may be the subject of future papers.

\section*{Acknowledgment}

The author gives special thanks to Dr.~Takuya Ooura for his valuable comments on this work. 
Moreover, the author would like to thank Prof.~Masaaki Sugihara 
for his cooperation on this work.
This work is supported by JSPS KAKENHI Grant Number 24760064.

\end{document}